\documentclass[a4paper,11pt]{amsart}
\usepackage{amssymb}
\usepackage[all]{xy}
\usepackage{mathptmx}

\title[A connectedness theorem]{A connectedness theorem over the spectrum of
\linebreak
a formal power series ring}

\author{Masayuki Kawakita}
\address{Research Institute for Mathematical Sciences, Kyoto University, Kyoto 606-8502, Japan}
\email{masayuki@kurims.kyoto-u.ac.jp}

\theoremstyle{plain}
\newtheorem{theorem}{Theorem}[section]
\newtheorem{proposition}[theorem]{Proposition}
\newtheorem{lemma}[theorem]{Lemma}
\newtheorem{corollary}[theorem]{Corollary}
\newtheorem{conjecture}[theorem]{Conjecture}
\newtheorem*{conjecture'acc}{Conjecture \ref{cnj:acc}$\mathbf{'}$}

\theoremstyle{definition}
\newtheorem{definition}[theorem]{Definition}

\theoremstyle{remark}
\newtheorem{remark}[theorem]{Remark}
\newtheorem*{acknowledgements}{Acknowledgements}

\newcommand{\bA}{\mathbb{A}}
\newcommand{\bN}{\mathbb{N}}
\newcommand{\bQ}{\mathbb{Q}}
\newcommand{\bR}{\mathbb{R}}
\newcommand{\bZ}{\mathbb{Z}}
\newcommand{\cD}{\mathcal{D}}
\newcommand{\cF}{\mathcal{F}}
\newcommand{\cI}{\mathcal{I}}
\newcommand{\cO}{\mathcal{O}}
\newcommand{\cQ}{\mathcal{Q}}
\newcommand{\fa}{\mathfrak{a}}
\newcommand{\fb}{\mathfrak{b}}

\newcommand{\fd}{\mathfrak{d}}
\newcommand{\fm}{\mathfrak{m}}
\newcommand{\fn}{\mathfrak{n}}
\newcommand{\fp}{\mathfrak{p}}
\newcommand{\fq}{\mathfrak{q}}
\newcommand{\sfa}{\mathsf{a}}
\newcommand{\sfb}{\mathsf{b}}
\newcommand{\sfc}{\mathsf{c}}
\newcommand{\sfd}{\mathsf{d}}
\newcommand{\sfq}{\mathsf{q}}
\newcommand{\sfI}{\mathsf{I}}

\newcommand{\id}{\mathrm{id}}
\newcommand{\pr}{\mathrm{pr}}

\newcommand{\ru}[1]{\lceil{#1}\rceil}
\newcommand{\agl}[1]{\langle{#1}\rangle}

\DeclareMathOperator{\codim}{codim}
\DeclareMathOperator{\Coef}{Coef}
\DeclareMathOperator{\Cosupp}{Cosupp}
\DeclareMathOperator{\Exc}{Exc}
\DeclareMathOperator{\mld}{mld}
\DeclareMathOperator{\mult}{mult}
\DeclareMathOperator{\Nklt}{Nklt}
\DeclareMathOperator{\ord}{ord}
\DeclareMathOperator{\Spec}{Spec}
\DeclareMathOperator{\Supp}{Supp}

\begin{document}
\begin{abstract}
We study the connectedness of the non-subklt locus over the spectrum of a formal power series ring. In dimension $3$, we prove the existence and normality of the smallest lc centre, and apply it to the ACC for minimal log discrepancies greater than $1$ on non-singular $3$-folds.
\end{abstract}

\maketitle

\section{Introduction}
The vanishing theorem by Kodaira \cite{Kd53} is one of the most basic tools in algebraic geometry in characteristic zero. It is reasonable to expect a vanishing theorem on excellent schemes, but it is annoyingly unknown besides the work on surfaces by Lipman \cite{Lp78}. Precisely, we are interested in the relative Kodaira vanishing for a birational morphism over the spectrum of a formal power series ring $R=K[[x_1,\ldots,x_d]]$ for a field $K$ of characteristic zero. We mean by an $R$-variety an integral separated scheme of finite type over $\Spec R$.

\begin{conjecture}
Let $f\colon Y\to X$ be a projective birational morphism of non-singular $R$-varieties and $L$ an $f$-ample divisor on $Y$. Then $R^if_*\cO_Y(K_{Y/X}+L)=0$ for $i\ge1$. Here the relative canonical divisor $K_{Y/X}$ is defined by the $0$-th Fitting ideal of $\Omega_{Y/X}$.
\end{conjecture}

We shall not deal with this algebraic conjecture. Instead, we study the \textit{connectedness lemma} by Shokurov \cite{Sh92} and Koll\'ar \cite{Kl92}, which is an important geometric application of the vanishing theorem in birational geometry. It claims for a proper morphism $f\colon Y\to X$, the fibrewise connectedness of the non-subklt locus of a subpair $(Y,\Delta)$ such that $\Delta$ is effective outside a locus in $X$ of codimension at least $2$ and such that $-(K_Y+\Delta)$ is $f$-nef and $f$-big. We shall verify it for a germ at a non-singular point of $X$ in the case when $f$ is isomorphic outside the central fibre (Theorem \ref{thm:connect}). Investigating further in dimension $3$, we obtain a desirable result on the smallest lc centre of a pair on a non-singular $R$-variety of dimension $3$.

\begin{theorem}\label{thm:centre}
Let $P\in(X,\fa)$ be a germ of an lc but not klt pair of a non-singular $R$-variety $X$ of dimension $3$ and an $\bR$-ideal $\fa$ on $X$. Then the smallest lc centre of $(X,\fa)$ exists and it is normal.
\end{theorem}

It is reduced to the case $X=\Spec R$ with $K$ an algebraically closed field $k$. Theorem \ref{thm:connect}, the fibrewise connectedness, is proved by approximating the effective $\bR$-divisor $f_*\Delta$ by an $\fm$-primary $\bR$-ideal $\fa\agl{l}$, where $\fm$ is the maximal ideal sheaf, such that the non-subklt locus of the subtriplet coming from $\fa\agl{l}$ coincides with the central fibre of the original non-subklt locus. The $\fa\agl{l}$ is descended to $\bA_k^d=\Spec k[x_1,\ldots,x_d]$, on which the connectedness lemma is applied. The existence of the smallest lc centre in Theorem \ref{thm:centre} is a corollary to Theorem \ref{thm:connect}. The hardest part of Theorem \ref{thm:centre} is the normality of the smallest lc centre $C$ which is a curve. We construct an ideal sheaf $\fn_a$ on the normalisation $C_Y$ of $C$ with $f_C\colon C_Y\to C$ which satisfies $f_{C*}\fn_a\subset\cO_C$ and $\cO_C/f_{C*}\fn_a\simeq f_{C*}\cO_{C_Y}/f_{C*}\fn_a$. Then we obtain the isomorphism $\cO_C\simeq f_{C*}\cO_{C_Y}$ meaning the normality of $C$.

Our motivation for excellent schemes stems from the notion of a generic limit of ideals due to de Fernex and Musta\c{t}\u{a} \cite{dFM09}. The generic limit was used to prove the ascending chain condition (ACC) for log canonical thresholds on non-singular varieties \cite{dFEM11}, the approach of which works even for the study of minimal log discrepancies \cite{K14}. We shall apply Theorem \ref{thm:centre} to the ACC conjecture for minimal log discrepancies by Shokurov \cite{Sh88}, \cite{Sh96} and McKernan \cite{M-MSRI} in the case of non-singular $3$-folds, and settle the part of minimal log discrepancies greater than $1$.

\begin{theorem}\label{thm:acc}
Fix subsets $I\subset(0,\infty)$ and $J\subset(1,3]$ both of which satisfy the descending chain condition. Then there exist finite subsets $I_0\subset I$ and $J_0\subset J$ such that if $P\in(X,\fa=\prod_j\fa_j^{r_j})$ is a germ of a pair of a non-singular variety $X$ of dimension $3$ and an $\bR$-ideal $\fa$ on $X$ with all $\fa_j$ non-trivial at $P$, all $r_j\in I$ and $\mld_P(X,\fa)\in J$, then all $r_j\in I_0$ and $\mld_P(X,\fa)\in J_0$.
\end{theorem}

The generic limit $\sfa$ of $\bR$-ideals $\fa_i$ on $P\in X=\Spec k[[x_1,\ldots,x_d]]$ is an $\bR$-ideal on $P_K\in X_K=\Spec K[[x_1,\ldots,x_d]]$ with a field extension $K$ of $k$. The ACC for minimal log discrepancies on non-singular $d$-folds is reduced to the stability $\mld_{P_K}(X_K,\sfa)=\mld_P(X,\fa_i)$ for general $i$. We prove it when $(X_K,\sfa)$ is a klt pair, or even a plt pair whose lc centre has an isolated singularity, by our previous arguments \cite{K13-1}, \cite{K13-2}. In dimension $3$, only the case when $(X_K,\sfa)$ has the smallest lc centre of dimension $1$ remains. In this case, the estimate $\mld_{P_K}(X_K,\sfa)\le1$ is derived from Theorem \ref{thm:centre}, which is enough to prove Theorem \ref{thm:acc}.

The structure of the paper is as follows. After reviewing the basics of singularities in Section \ref{sec:singularities}, we study the connectedness of the non-subklt locus and establish Theorem \ref{thm:centre} in Section \ref{sec:connect}. We discuss the ACC for minimal log discrepancies from the point of view of generic limits in Section \ref{sec:acc}. The stability of minimal log discrepancies in the klt and plt cases is shown in Section \ref{sec:kltplt}. Theorem \ref{thm:acc} is completed in Section \ref{sec:threefold}. The appendix exposing generic limits is attached.

Throughout this paper, $k$ is an algebraically closed field of characteristic zero.

\section{Singularities}\label{sec:singularities}
We review the basics of singularities in birational geometry. A good reference is \cite{Kl12}. A \textit{variety} is an integral separated scheme of finite type over $\Spec k$. A \textit{germ} of a scheme is considered at a closed point.

An $\bR$-\textit{ideal} on a noetherian scheme $X$ is a formal product $\fa=\prod_j\fa_j^{r_j}$ of finitely many coherent ideal sheaves $\fa_j$ on $X$ with positive real exponents $r_j$. The $\fa$ to the \textit{power} of $t>0$ is $\fa^t:=\prod_j\fa_j^{tr_j}$. The \textit{co-support} $\Cosupp\fa$ of $\fa$ is the union of all $\Supp\cO_X/\fa_j$. The \textit{pull-back} of $\fa$ by a morphism $Y\to X$ is $\fa\cO_Y:=\prod_j(\fa_j\cO_Y)^{r_j}$. The $\bR$-ideal $\fa$ is said to be \textit{invertible} if all $\fa_j$ are invertible. In this case, if in addition $X$ is normal, then the $\bR$-divisor $A=\sum_jr_jA_j$ with $\fa_j=\cO_X(-A_j)$ is called the $\bR$-divisor \textit{defined by} $\fa$.

Let $Z$ be an irreducible closed subset of $X$. We write $\eta_Z$ for the generic point of $Z$. The \textit{order} of $\fa$ along $Z$ is $\ord_Z\fa=\sum_jr_j\ord_Z\fa_j$, where $\ord_Z\fa_j$ is the maximal $\nu\in\bN\cup\{+\infty\}$ satisfying $\fa_j\cO_{X,\eta_Z}\subset\cI_Z^\nu\cO_{X,\eta_Z}$ for the ideal sheaf $\cI_Z$ of $Z$.

We treat a \textit{triplet} $(X,\Delta,\fa)$ which consists of a normal variety $X$, an effective $\bR$-divisor $\Delta$ on $X$ such that $K_X+\Delta$ is an $\bR$-Cartier $\bR$-divisor, and an $\bR$-ideal $\fa=\prod_j\fa_j^{r_j}$ on $X$. A prime divisor $E$ on a normal variety $Y$ with a birational morphism $f\colon Y\to X$ is called a divisor \textit{over} $X$, and the closure $\overline{f(E)}$ of the image on $X$ is called the \textit{centre} of $E$ on $X$ and denoted by $c_X(E)$. We denote by $\cD_X$ the set of all divisors over $X$. The \textit{log discrepancy} of $E$ with respect to $(X,\Delta,\fa)$ is
\begin{align*}
a_E(X,\Delta,\fa):=1+\ord_EK_{Y/(X,\Delta)}-\ord_E\fa,
\end{align*}
where $K_{Y/(X,\Delta)}:=K_Y-f^*(K_X+\Delta)$ and $\ord_E\fa:=\ord_E\fa\cO_Y$. Note that $c_X(E)$ and $a_E(X,\Delta,\fa)$ are determined by the valuation on the function field of $X$ given by $E$.

For an irreducible closed subset $Z$ of $X$, the \textit{minimal log discrepancy} of $(X,\Delta,\fa)$ at $\eta_Z$ is
\begin{align*}
\mld_{\eta_Z}(X,\Delta,\fa):=\inf\{a_E(X,\Delta,\fa)\mid E\in\cD_X,\ c_X(E)=Z\}.
\end{align*}
It is either a non-negative real number or $-\infty$. We say that $E\in\cD_X$ \textit{computes} $\mld_{\eta_Z}(X,\Delta,\fa)$ if $c_X(E)=Z$ and $a_E(X,\Delta,\fa)=\mld_{\eta_Z}(X,\Delta,\fa)$ (or is negative when $\mld_{\eta_Z}(X,\Delta,\fa)=-\infty$). It is often reduced to the case when $Z$ is a closed point by the relation $\mld_{\eta_Z}(X,\Delta,\fa)=\mld_P(X,\Delta,\fa)-\dim Z$ for a general closed point $P\in Z$ (cf.\ \cite[Proposition 2.1]{Am99}).

The triplet $(X,\Delta,\fa)$ is said to be \textit{log canonical} (\textit{lc}) (resp.\ \textit{Kawamata log terminal} (\textit{klt})) if $a_E(X,\Delta,\fa)\ge0$ (resp.\ $>0$) for all $E\in\cD_X$. It is said to be \textit{purely log terminal} (\textit{plt}) (resp.\ \textit{canonical}, \textit{terminal}) if $a_E(X,\Delta,\fa)>0$ (resp.\ $\ge1$, $>1$) for all exceptional $E\in\cD_X$. The log canonicity of $(X,\Delta,\fa)$ about $P\in X$ is equivalent to $\mld_P(X,\Delta,\fa)\ge0$. Let $Y$ be a normal variety with a birational morphism to $X$. A centre $c_Y(E)$ with $a_E(X,\Delta,\fa)\le0$ is called a \textit{non-klt centre} on $Y$ of $(X,\Delta,\fa)$. The union of all non-klt centres on $Y$ is called the \textit{non-klt locus} on $Y$ and denoted by $\Nklt_Y(X,\Delta,\fa)$. When we say just a non-klt centre or the non-klt locus, we mean that it is on $X$.

A \textit{log resolution} of $(X,\Delta,\fa)$ is a projective morphism $f\colon Y\to X$ from a non-singular variety $Y$ such that (i) $\Exc f$ is a divisor, (ii) $\fa\cO_Y$ is invertible, (iii) $\Exc f\cup\Supp\Delta_Y\cup\Cosupp\fa\cO_Y$ is a simple normal crossing (snc) divisor, where $\Delta_Y$ is the strict transform of $\Delta$, and (iv) $f$ is isomorphic on the locus $U$ in $X$ with $U$ non-singular, $\fa|_U$ invertible and $\Supp\Delta|_U\cup\Cosupp\fa|_U$ snc. A \textit{stratum} (resp.\ an \textit{open stratum}) of an snc divisor $\sum_{i\in I}E_i$ is an irreducible component of $\bigcap_{i\in J}E_i$ (resp.\ $\bigcap_{i\in J}E_i\setminus\bigcup_{i\not\in J}E_i$) for a subset $J$ of $I$.

By allowing a not necessarily effective $\bR$-divisor $\Delta$, one can consider a \textit{subtriplet} $(X,\Delta,\fa=\prod_j\fa_j^{r_j})$. The notions of lc (resp.\ klt) singularities are extended for subtriplets, in which we say \textit{sublc} (resp.\ \textit{subklt}) singularities. Let $f\colon Y\to X$ be a birational morphism from a non-singular variety $Y$ such that $\Exc f$ is a divisor $\sum_iE_i$. The \textit{weak transform} on $Y$ of $\fa$ is the $\bR$-ideal $\fa_Y=\prod_j\fa_{jY}^{r_j}$ with $\fa_{jY}=\fa_j\cO_Y(\sum_i(\ord_{E_i}\fa_j)E_i)$.

\begin{definition}
Notation as above. The \textit{pull-back} of $(X,\Delta,\fa)$ by $f$ is the subtriplet $(Y,\Delta_Y,\fa_Y)$ where $\Delta_Y=-K_{Y/(X,\Delta)}+\sum_{ij}(r_j\ord_{E_i}\fa_j)E_i$.
\end{definition}

The $(X,\Delta,\fa)$ is sublc (resp.\ subklt) if and only if so is $(Y,\Delta_Y,\fa_Y)$. We use the notation $\Nklt_Y(X,\Delta,\fa)$ also for the non-subklt locus on $Y$ of a subtriplet $(X,\Delta,\fa)$.

These definitions are extended on schemes over a field $K$ of characteristic zero and even over a formal power series ring $R=K[[x_1,\ldots,x_d]]$ by the existence of log resolutions due to Hironaka \cite{Hi64} and Temkin \cite{T08}, \cite{T-Arxiv}. This extension is studied by de Fernex, Ein and Musta\c{t}\u{a} \cite{dFEM11}, \cite{dFM09}. We mean by an \textit{$R$-variety} an integral separated scheme of finite type over $\Spec R$.

The canonical divisor $K_X$ on a normal $R$-variety $X$ is defined by the isomorphism $\cO_X(K_X)|_U\simeq\bigwedge^r\Omega'_{X/K}|_U$ on the non-singular locus $U$ of $X$, where $\Omega'_{X/K}$ is the \textit{sheaf of special differentials} in \cite{dFEM11} and $r$ is its rank. The relative canonical divisor is well understood for a birational morphism of non-singular $R$-varieties.

\begin{lemma}[{\cite[Remark A.12]{dFEM11}}]\label{lem:extension}
Let $Y\to X$ be a proper birational morphism of non-singular $R$-varieties. Then $K_{Y/X}$ is the effective divisor defined by the $0$-th Fitting ideal of $\Omega_{Y/X}$. In particular, $K_{Y/X}$ is independent of the structure of $X$ as an $R$-variety.
\end{lemma}

The log discrepancies are preserved by field extensions and completions.

\begin{corollary}\label{cor:extension}
Let $Y\to X$ be as in Lemma \textup{\ref{lem:extension}}. Take an $R'$-variety $X'$ as in \textup{(\ref{itm:extension_field})}, \textup{(\ref{itm:extension_compl})} or \textup{(\ref{itm:extension_base})} below and set a morphism $Y'=Y\times_XX'\to X'$ of $R'$-varieties.
\begin{enumerate}
\item\label{itm:extension_field}
$X'$ is a component of $X\times_{\Spec R}\Spec R'$ with $R'=\widehat{R\otimes_KK'}$ for a field extension $K'$ of $K$.
\item\label{itm:extension_compl}
$X'=\Spec\widehat{\cO_{X,P}}$ for a germ $P\in X$, which admits the structure of an $R'$-variety for a suitable $R'=K'[[x_1,\ldots,x_{d'}]]$ by Cohen's structure theorem \cite{C46}.
\item\label{itm:extension_base}
$X'=X$ with another structure morphism $X\to\Spec R'$.
\end{enumerate}
Then $K_{Y'/X'}$ is the pull-back of $K_{Y/X}$. In particular, for an $\bR$-ideal $\fa$ on $X$, a divisor $E$ over $X$ and a germ $P\in X$, one has $a_{E'}(X',\fa\cO_{X'})=a_E(X,\fa)$ for a component $E'$ of $E\times_XX'$ and $\mld_{P'}(X',\fa\cO_{X'})=\mld_P(X,\fa)$ for a point $P'$ of $P\times_XX'$.
\end{corollary}

This is by the regularity of the morphism $X'\to X$. The cases (\ref{itm:extension_field}) and (\ref{itm:extension_compl}) for $R=K$ are stated in \cite[Lemma 2.14, Propositions 2.11, A.14]{dFEM11} even for a normal ($\bQ$-Gorenstein) $K$-variety $X$.

Suppose that $(X,\Delta,\fa)$ is lc. Then a non-klt centre (on $X$) of $(X,\Delta,\fa)$ is often called an \textit{lc centre}. An lc centre which is minimal with respect to inclusions is called a \textit{minimal lc centre}. When we work over a germ $P\in X$, the following definition makes sense.

\begin{definition}
Let $P\in(X,\Delta,\fa)$ be a germ of an lc triplet. The \textit{smallest lc centre} is an lc centre of $(X,\Delta,\fa)$ passing through $P$ contained in every lc centre passing through $P$.
\end{definition}

If $X$ is a variety, then the smallest lc centre exists and it is normal \cite[Theorem 9.1]{F11}. It is, however, unknown for $R$-varieties. Theorem \ref{thm:centre} states that this is the case when $X$ is a non-singular $R$-variety of dimension $3$.

\section{The smallest lc centre on a threefold}\label{sec:connect}
This section is devoted to the proof of Theorem \ref{thm:centre}. We work over a germ $P\in X$ of an $R$-variety with $R=K[[x_1,\ldots,x_d]]$. The maximal ideal sheaf of $P\in X$ is denoted by $\fm$. When we discuss on the spectrum of a noetherian ring, we identify an ideal in the ring with its coherent ideal sheaf.

\subsection{A connectedness theorem}
We prove a connectedness theorem over $X$.

\begin{theorem}\label{thm:connect}
Let $P\in(X,\fa)$ be a germ of a pair on a non-singular $R$-variety $X$ and $f\colon Y\to X$ a proper birational morphism of non-singular $R$-varieties which is isomorphic outside $P$. Let $\Delta$ be an $\bR$-divisor on $Y$ with $f_*\Delta\ge0$ such that $-(K_Y+\Delta)$ is $f$-nef. Then $\Nklt_Y(Y,\Delta,\fa\cO_Y)\cap f^{-1}(P)$ is connected.
\end{theorem}

We extract the case $\Delta=-K_{Y/X}$.

\begin{corollary}\label{cor:connect}
Let $P\in(X,\fa)$ be a germ of a pair on a non-singular $R$-variety $X$ and $f\colon Y\to X$ a proper birational morphism of non-singular $R$-varieties which is isomorphic outside $P$. Then $\Nklt_Y(X,\fa)\cap f^{-1}(P)$ is connected.
\end{corollary}

The statement for $R=k$ is a special case of the connectedness lemma by Sho-
\linebreak
kurov and Koll\'ar \cite[Theorem 17.4]{Kl92}. It settles the case when $\fa$ is $\fm$-primary and $\Delta$ is $f$-exceptional. Write $\fa=\prod_j\fa_j^{r_j}$.

\begin{lemma}\label{lem:connect}
\begin{enumerate}
\item\label{itm:connect_reduction}
In order to prove Theorem \textup{\ref{thm:connect}}, one may assume that $X=\Spec R$ with $K=k$, $f$ is projective and $\Delta$ is $f$-exceptional.
\item\label{itm:connect_primary}
Theorem \textup{\ref{thm:connect}} holds in the case when $X=\Spec R$ with $K=k$, $f$ is projective, $\Delta$ is $f$-exceptional and all $\fa_j$ are $\fm$-primary ideals.
\end{enumerate}
\end{lemma}

\begin{proof}
(\ref{itm:connect_reduction})\
Take an isomorphism $\widehat{\cO_{X,P}}\simeq K'[[x_1,\ldots,x_{d'}]]$ with $K'=\cO_{X,P}/\fm$ by Cohen's structure theorem and set $R'=k[[x_1,\ldots,x_{d'}]]$ for the algebraic closure $k$ of $K'$. Because the base change $\Spec R'\to X$ commutes with taking the non-subklt locus by Corollary \ref{cor:extension}, we may assume $X=\Spec R$ with $K=k$ (the $d$ may be changed). By the flattening theorem of Raynaud and Gruson \cite[Th\'eor\`eme $\textrm{1}^{\textrm{re}}$ 5.2.2]{RG71}, there exists a projective morphism $f'\colon Y'\to X$ from a non-singular $R$-variety $Y'$ which is isomorphic outside $P$ and factors through $f$. Replacing $(Y,\Delta)$ with its pull-back on $Y'$, we may assume that $f$ is projective. The $\Delta':=\Delta-f^*f_*\Delta$ is $f$-exceptional. Take an invertible $\bR$-ideal $\fd$ on $X$ which defines the $\bR$-divisor $f_*\Delta\ge0$. Then $\Nklt_Y(Y,\Delta,\fa\cO_Y)=\Nklt_Y(Y,\Delta',\fa\fd\cO_Y)$. Replacing $\Delta$ with $\Delta'$ and $\fa$ with $\fa\fd$, we may assume that $\Delta$ is $f$-exceptional.

(\ref{itm:connect_primary})\
We use the notation $\bar{R}=k[x_1,\ldots,x_d]$ and $\bA_k^d=\Spec\bar{R}$ with origin $\bar{P}$. By Proposition \ref{prp:descendible}, $f$ is the base change of a projective morphism $\bar{f}\colon\bar{Y}\to\bA_k^d$ and $\fa$ is the pull-back of the $\bR$-ideal $\bar\fa=\prod_j(\fa_j\cap\bar{R})^{r_j}$. Then $f^{-1}(P)\simeq\bar{f}^{-1}(\bar{P})$ and $\Delta$ is the base change of an $\bar{f}$-exceptional $\bR$-divisor $\bar\Delta$ such that $-(K_{\bar{Y}}+\bar\Delta)$ is $\bar{f}$-nef. Thus $ f^{-1}(P)\supset\Nklt_Y(Y,\Delta,\fa\cO_Y)\simeq\Nklt_{\bar{Y}}(\bar{Y},\bar\Delta,\bar\fa\cO_{\bar{Y}})$, which is connected by \cite[Theorem 17.4]{Kl92}.
\end{proof}

We take a log resolution $q\colon W\to Y$ of $(Y,\Delta,\fa\fm\cO_Y)$ and set the composition $g=f\circ q\colon W\to X$. We fix $\varepsilon>0$ such that
\begin{align}\label{eqn:F}
F:=\Nklt_W(Y,\Delta,\fa\cO_Y)=\Nklt_W(Y,\Delta,\fa^{1+\varepsilon}\cO_Y).
\end{align}
We approximate $\fa$ by an $\fm$-primary $\bR$-ideal
\begin{align}\label{eqn:al}
\fa\agl{l}:=\prod_j(\fa_j+\fm^l)^{r_j(1+\varepsilon)}
\end{align}
with $l\in\bN$.

Consider an irreducible component $D$ of $F\cap g^{-1}(P)$ with $\codim_WD=2$, and let $E_D\subset g^{-1}(P)$ and $F_D\subset F$ be the prime divisors such that $D\subset E_D\cap F_D$. We build a tower of blow-ups
\begin{align}\label{eqn:tower}
\cdots\to W_i\xrightarrow{g_i}W_{i-1}\to\cdots\to W_0=W
\end{align}
as follows. Set $W_0:=W$, $E_0:=E_D$ and $F_0:=F_D$. We construct inductively the blow-up $g_i\colon W_i\to W_{i-1}$ along $D$ for $i=1$ (resp.\ along $E_{i-1}\cap F_{i-1}$ for $i\ge2$), and set $E_i$ as the exceptional divisor of $g_i$, and $F_i$ as the strict transform on $W_i$ of $F_D$. The composition $g_1\circ\cdots\circ g_i$ is denoted by $h_i\colon W_i\to W$.

\begin{lemma}\label{lem:tower}
\begin{enumerate}
\item\label{itm:tower_ld}
$a_{E_i}(Y,\Delta,\fa^{1+\varepsilon}\cO_Y)\le a_{E_D}(Y,\Delta,\fa^{1+\varepsilon}\cO_Y)-i\varepsilon\ord_{F_D}\fa$.
\item\label{itm:tower_ideal}
$h_{i*}\cO_{W_i}(-aE_i)\subset\cO_W(-aE_D)+\cO_W(-F_D)$ for any $a\in\bN$.
\end{enumerate}
\end{lemma}

\begin{proof}
The (\ref{itm:tower_ld}) is just a computation using $a_{F_D}(Y,\Delta,\fa\cO_Y)\le0$. The (\ref{itm:tower_ideal}) is from $h_{i*}\cO_{W_i}(-aE_i)\cdot\cO_{F_D}\subset h_{i*}\cO_{F_i}(-aE_i|_{F_i})=\cO_{F_D}(-aE_D|_{F_D})$ via $F_i\simeq F_D$.
\end{proof}

\begin{lemma}\label{lem:fat}
Suppose that $(Y,\Delta)$ is klt outside $f^{-1}(P)$. Then there exists $l$ such that $\Nklt_W(Y,\Delta,\fa\agl{l}\cO_Y)=F\cap g^{-1}(P)$.
\end{lemma}

\begin{proof}
By (\ref{eqn:F}), (\ref{eqn:al}) and the assumption, $\Nklt_W(Y,\Delta,\fa\agl{l}\cO_Y)\subset F\cap g^{-1}(P)$ for any $l$. Thus it suffices to prove that for every irreducible component $D$ of $F\cap g^{-1}(P)$, there exists $l_D$ such that $D$ is a non-subklt centre on $W$ of $(Y,\Delta,\fa\agl{l}\cO_Y)$ for any $l\ge l_D$. If $\codim_WD=1$, then we may take any $l_D$ such that $l_D\ord_D\fm\ge\ord_D\fa_j$ for all $j$. If $\codim_WD=2$, then $\ord_{F_D}\fa>0$ and we take the tower of blow-ups in (\ref{eqn:tower}). By Lemma \ref{lem:tower}(\ref{itm:tower_ld}), we have $a_{E_i}(Y,\Delta,\fa^{1+\varepsilon}\cO_Y)\le0$ whenever $a_{E_D}(Y,\Delta,\fa^{1+\varepsilon}\cO_Y)\le i\varepsilon\ord_{F_D}\fa$. Fix such $i$ and take $l_D$ such that $l_D\ord_{E_i}\fm\ge\ord_{E_i}\fa_j$ for all $j$. Then for $l\ge l_D$, $a_{E_i}(Y,\Delta,\fa\agl{l}\cO_Y)=a_{E_i}(Y,\Delta,\fa^{1+\varepsilon}\cO_Y)\le0$, so $D=c_W(E_i)$ is a non-subklt centre on $W$ of $(Y,\Delta,\fa\agl{l}\cO_Y)$.
\end{proof}

\begin{proof}[Proof of Theorem \textup{\ref{thm:connect}}]
After the reduction in Lemma \ref{lem:connect}(\ref{itm:connect_reduction}), we take $l$ in Lemma \ref{lem:fat}. Then $\Nklt_Y(Y,\Delta,\fa\cO_Y)\cap f^{-1}(P)=q(F\cap g^{-1}(P))=q(\Nklt_W(Y,\Delta,\fa\agl{l}\cO_Y))=\Nklt_Y(Y,\Delta,\fa\agl{l}\cO_Y)$. Apply Lemma \ref{lem:connect}(\ref{itm:connect_primary}) to $(Y,\Delta,\fa\agl{l}\cO_Y)$.
\end{proof}

In our proof of Theorem \ref{thm:connect}, we do not know a relative vanishing for $g\colon W\to X$. Instead, we consider a log resolution $f_l\colon Y_l\to X$ of $(X,\fa\agl{l}\fm)$ which factors through $f$, and let $p_l\colon Y_l\to Y$ be the induced morphism. The $l$ is not fixed here. The $f_l$ is isomorphic outside $P$. Let $(Y_l,\Delta_l,\cO_{Y_l})$ be the pull-back of $(X,0,\fa\agl{l})$. Then we have a vanishing involving $\Delta_l$.

\begin{lemma}\label{lem:vanishing}
Let $f_l=f\circ p_l\colon Y_l\to Y\to X$ be as above. Write $\ru{-\Delta_l}=P_l-N_l$ by effective divisors $P_l$ and $N_l$ with no common divisors. Then
\begin{align*}
R^1f_*(p_{l*}\cO_{Y_l}(-N_l))=0.
\end{align*}
\end{lemma}

\begin{proof}
The sheaf $R^1f_*(p_{l*}\cO_{Y_l}(-N_l))$ is supported in $P$. Set $\widehat{\cO_{X,P}}\simeq K'[[x_1,\ldots,x_{d'}]]$ and $R'=k[[x_1,\ldots,x_{d'}]]$ for the algebraic closure $k$ of $K'$, then $R'$ is faithfully flat over $\cO_{X,P}$. Hence taking the base change to $\Spec R'$, one can reduce to the case $X=\Spec R$ with $K=k$ by \cite[Proposition III.1.4.15]{EGA} and Corollary \ref{cor:extension}. By Proposition \ref{prp:descendible}, $f_l$ is the base change of a projective morphism $\bar{f}_l\colon\bar{Y}_l\to\bA_k^d$. The $\fa\agl{l}$ is the pull-back of an $\bR$-ideal $\bar\fa\agl{l}$ on $\bA_k^d$, and $\Delta_l$ is the base change of the $\bR$-divisor $\bar\Delta_l$ on $\bar{Y}_l$ such that $(\bar{Y}_l,\bar\Delta_l,\cO_{\bar{Y}_l})$ is the pull-back of $(\bA_k^d,0,\bar\fa\agl{l})$.

Kawamata--Viehweg vanishing theorem \cite{Km82}, \cite{Vi82} implies $R^1\bar{f}_{l*}\cO_{\bar{Y}_l}(\ru{-\bar\Delta_l})=0$. Since $X\to\bA_k^d$ is flat, this is base-changed to $R^1f_{l*}\cO_{Y_l}(\ru{-\Delta_l})=0$ by \cite[Proposition III.1.4.15]{EGA}. Thus, applying $f_{l*}$ to the exact sequence
\begin{align*}
0\to\cO_{Y_l}(P_l-N_l)\to\cO_{Y_l}(P_l)\to\cO_{N_l}(P_l|_{N_l})\to0,
\end{align*}
we obtain the surjection $\cO_X=f_{l*}\cO_{Y_l}(P_l)\twoheadrightarrow f_{l*}\cO_{N_l}(P_l|_{N_l})$. This homomorphism is factored as $\cO_X\to f_{l*}\cO_{N_l}\hookrightarrow f_{l*}\cO_{N_l}(P_l|_{N_l})$, so we have the surjection $\cO_X\twoheadrightarrow f_{l*}\cO_{N_l}$. Moreover, we have the base change $R^1f_{l*}\cO_{Y_l}=0$ of the vanishing $R^1\bar{f}_{l*}\cO_{\bar{Y}_l}=0$. Hence applying $f_{l*}$ to the exact sequence
\begin{align*}
0\to\cO_{Y_l}(-N_l)\to\cO_{Y_l}\to\cO_{N_l}\to0,
\end{align*}
we obtain $R^1f_{l*}\cO_{Y_l}(-N_l)=0$.

Leray spectral sequence $R^pf_*(R^qp_{l*}\cO_{Y_l}(-N_l))\Rightarrow R^{p+q}f_{l*}\cO_{Y_l}(-N_l)$ gives an injection $R^1f_*(p_{l*}\cO_{Y_l}(-N_l))\hookrightarrow R^1f_{l*}\cO_{Y_l}(-N_l)$, so $R^1f_*(p_{l*}\cO_{Y_l}(-N_l))=0$.
\end{proof}

\subsection{Propositions in an arbitrary dimension}
We prepare two auxiliary propositions.

It is easy to see that a minimal lc centre of codimension $1$ is normal.

\begin{proposition}\label{prp:centre}
Let $(X,\fa)$ be a pair on a non-singular $R$-variety $X$, and $S$ the union of all non-klt centres of codimension $1$ of $(X,\fa)$. Then every irreducible component of the non-normal locus of $S$ is a non-klt centre of $(X,\fa)$.
\end{proposition}

\begin{proof}
Since $S$ is Cohen--Macaulay, an irreducible component $C$ of the non-normal locus of $S$ has $\codim_XC=2$ and $\mult_{\eta_C}S\ge2$. Let $E$ be the divisor over $X$ obtained at $\eta_C$ by the blow-up of $X$ along $C$. Then $a_E(X,\fa)=2-\ord_E\fa\le2-\mult_{\eta_C}S\le0$, so $C=c_X(E)$ is a non-klt centre of $(X,\fa)$.
\end{proof}

We can perturb $\fa$ to reduce to the case when every lc centre is minimal.

\begin{proposition}\label{prp:perturb}
Let $(X,\fa)$ be an lc pair on a klt $R$-variety $X$. Then there exists an $\bR$-ideal $\fa'$ forming an lc pair $(X,\fa')$ such that a minimal lc centre of $(X,\fa)$ is an lc centre of $(X,\fa')$ and vice versa.
\end{proposition}

\begin{proof}
Let $\{Z_i\}_i$ be the set of all minimal lc centres of $(X,\fa=\prod_j\fa_j^{r_j})$. For each $Z_i$, fix $E_i\in\cD_X$ computing $\mld_{\eta_{Z_i}}(X,\fa)=0$. Let $\cI_Z$ be the ideal sheaf of $Z=\bigcup_iZ_i$, and take an integer $l$ such that $l\ord_{E_i}\cI_Z\ge\ord_{E_i}\fa_j$ for all $i$, $j$. Then $(X,\fa':=\prod_j(\fa_j+\cI_Z^l)^{r_j})$ is lc, and $Z_i$ is an lc centre of $(X,\fa')$ by $\ord_{E_i}\fa'=\ord_{E_i}\fa$. On the other hand, every lc centre of $(X,\fa')$ is an lc centre of $(X,\fa)$ contained in $\Cosupp\fa'=Z$, so it equals some $Z_i$.
\end{proof}

\subsection{The smallest lc centre on a threefold}
We proceed to the proof of Theorem \ref{thm:centre}. We may assume that $P$ is not an lc centre of $(X,\fa)$. By Proposition \ref{prp:perturb}, we may assume that every lc centre of $(X,\fa)$ is minimal.

The existence of the smallest lc centre is a consequence of Corollary \ref{cor:connect}.

\begin{proof}[Proof of the existence of the smallest lc centre]
Let $\{Z_i\}_i$ be the set of all lc centres of $(X,\fa)$, which are assumed to be minimal. Proposition \ref{prp:centre} implies that $Z=\bigcup_iZ_i$ is non-singular outside $P$. Thus we have an embedded resolution $f\colon Y\to X$ of singularities of $Z$, in which $f$ is isomorphic outside $P$ and induces $f_Z\colon\bigsqcup_iZ_{iY}\to Z$ for the strict transform $Z_{iY}$ of $Z_i$. By Corollary \ref{cor:connect}, $f_Z^{-1}(P)=\Nklt_Y(X,\fa)\cap f^{-1}(P)$ is connected, that is, there exists only one lc centre of $(X,\fa)$.
\end{proof}

\begin{remark}\label{rmk:point}
The above proof shows that if $Z$ is the smallest lc centre of $(X,\fa)$, then its normalisation $Z^\nu\to Z$ is a homeomorphism.
\end{remark}

To complete Theorem \ref{thm:centre}, we must prove that the unique lc centre of $(X,\fa)$ is normal. If it is a surface, then it is normal by Proposition \ref{prp:centre}. Thus, we may assume that $(X,\fa)$ has the unique lc centre $C$ which is a curve. We have an embedded resolution $f\colon Y\to X$ of singularities of $C$, in which $f$ is isomorphic outside $P$ and induces the normalisation $f_C\colon C_Y\to C$ for the strict transform $C_Y$ of $C$. Note that $f_C^{-1}(P)$ consists of one point, say $P_Y$, by Remark \ref{rmk:point}. We let $\fn$ denote the maximal ideal sheaf of $P_Y\in Y$. Then we take a log resolution $q\colon W\to Y$ of $(Y,\fa\fm\cO_Y\cdot\fn)$ and set the composition $g=f\circ q\colon W\to X$.

We fix $\varepsilon$ in (\ref{eqn:F}) for $\Delta=-K_{Y/X}$, that is, $F=\Nklt_W(X,\fa)=\Nklt_W(X,\fa^{1+\varepsilon})$. For the $\fa\agl{l}$ in (\ref{eqn:al}), we consider a log resolution $f_l\colon Y_l\to X$ of $(X,\fa\agl{l}\fm)$ which factors through $f$ as $f_l=f\circ p_l$. We extend Lemma \ref{lem:vanishing}.

\begin{lemma}\label{lem:refine}
Let $f$ and $f_l=f\circ p_l$ be as above. Then for an arbitrary ideal sheaf $\cI$ on $Y$ containing $p_{l*}\cO_{Y_l}(-N_l)$, with $N_l$ in Lemma \textup{\ref{lem:vanishing}}, one has $R^1f_*\cI=0$.
\end{lemma}

\begin{proof}
By (\ref{eqn:F}) for $\Delta=-K_{Y/X}$ and (\ref{eqn:al}), we see $p_l(\Supp N_l)=\Nklt_Y(X,\fa\agl{l})\subset q(F\cap g^{-1}(P))=C_Y\cap f^{-1}(P)=P_Y$, whence the cokernel $\cQ$ of the natural injection $p_{l*}\cO_{Y_l}(-N_l)\hookrightarrow\cI$ is a skyscraper sheaf. In particular, $R^1f_*\cQ=0$. Apply $f_*$ to the exact sequence
\begin{align*}
0\to p_{l*}\cO_{Y_l}(-N_l)\to\cI\to\cQ\to0.
\end{align*}
By Lemma \ref{lem:vanishing} and $R^1f_*\cQ=0$, we obtain $R^1f_*\cI=0$.
\end{proof}

We fix an irreducible component $D$ of $F\cap q^{-1}(P_Y)$, which is a curve, and let $E_D\subset q^{-1}(P_Y)$ and $F_D\subset F$ be the prime divisors such that $D\subset E_D\cap F_D$. We derive a vanishing for ideal sheaves on $Y$ close to that of $C_Y$.

\begin{lemma}\label{lem:approx_van}
$R^1f_*(q_*(\cO_W(-aE_D)+\cO_W(-F_D)))=0$ for any $a\in\bN$.
\end{lemma}

\begin{proof}
Take the tower of blow-ups in (\ref{eqn:tower}). For fixed $a$, choose $i\in\bN$ such that $a_{E_D}(X,\fa^{1+\epsilon})-i\epsilon\ord_{F_D}\fa\le-a$. Then Lemma \ref{lem:tower} for $\Delta=-K_{Y/X}$ shows
\begin{align}\label{eqn:ideal}
h_{i*}\cO_{W_i}(\ru{a_{E_i}(X,\fa^{1+\varepsilon})}E_i)\subset h_{i*}\cO_{W_i}(-aE_i)\subset\cO_W(-aE_D)+\cO_W(-F_D).
\end{align}
Take $l$ such that $l\ord_{E_i}\fm\ge\ord_{E_i}\fa_j$ for all $j$. Then,
\begin{align}\label{eqn:ld}
a_{E_i}(X,\fa^{1+\varepsilon})=a_{E_i}(X,\fa\agl{l}).
\end{align}
For this $l$, we take a log resolution $f_l\colon Y_l\to X$ of $(X,\fa\agl{l}\fm)$ which factors through $f$, such that $c_{Y_l}(E_i)$ is a divisor. Then by (\ref{eqn:ideal}) and (\ref{eqn:ld}), one can apply Lemma \ref{lem:refine} to $\cI=q_*(\cO_W(-aE_D)+\cO_W(-F_D))$.
\end{proof}

We set the ideal sheaf $\fn_a$ on $C_Y$ as
\begin{align*}
\fn_a:=q_*(\cO_W(-aE_D)+\cO_W(-F_D))\cdot\cO_{C_Y}.
\end{align*}

\begin{lemma}\label{lem:topology}
There exists $a$ such that $f_{C*}\fn_a\subset\cO_C$.
\end{lemma}

\begin{proof}
Note that $\fn\cO_{C_Y}$ is an invertible ideal sheaf on $C_Y$. Set $n=\ord_{E_D}\fn$, then
\begin{align}\label{eqn:fibre}
\fn_{nl}\subset q_*\cO_{F_D}(-nlE_D|_{F_D})=\fn^l\cO_{C_Y}
\end{align}
for any $l$. Take an $f$-exceptional divisor $A\ge0$ on $Y$ such that $-A$ is $f$-ample and set $\cO_{C_Y}(-A|_{C_Y})=\fn^t\cO_{C_Y}$. By Serre vanishing theorem \cite[Th\'eor\`eme III.2.2.1]{EGA}, there exists $m_0$ such that $R^1f_*\cI_{C_Y}(-mA)=0$ for any $m\ge m_0$, where $\cI_{C_Y}$ is the ideal sheaf of $C_Y$ on $Y$. Then we have the surjection $f_*\cO_Y(-mA)\twoheadrightarrow f_{C*}\cO_{C_Y}(-mA|_{C_Y})=f_{C*}\fn^{tm}\cO_{C_Y}$, which provides
\begin{align}\label{eqn:Serre}
f_{C*}\fn^{tm}\cO_{C_Y}=f_*\cO_Y(-mA)\cdot\cO_C\subset\cO_C.
\end{align}
Combining (\ref{eqn:fibre}) and (\ref{eqn:Serre}), we obtain $f_{C*}\fn_{ntm}\subset f_{C*}\fn^{tm}\cO_{C_Y}\subset\cO_C$ for $m\ge m_0$.
\end{proof}

\begin{proof}[Proof of the normality of $C$]
Applying $f_*$ to the exact sequence
\begin{align*}
0\to q_*(\cO_W(-aE_D)+\cO_W(-F_D))\to\cO_Y\to\cO_{C_Y}/\fn_a\to0
\end{align*}
and using Lemma \ref{lem:approx_van}, we obtain the surjection $\cO_X\twoheadrightarrow f_{C*}(\cO_{C_Y}/\fn_a)$. This homomorphism is factored as $\cO_X\twoheadrightarrow\cO_C/f_{C*}\fn_a\cap\cO_C\hookrightarrow f_{C*}\cO_{C_Y}/f_{C*}\fn_a\hookrightarrow f_{C*}(\cO_{C_Y}/\fn_a)$, so we have an isomorphism $\cO_C/f_{C*}\fn_a\cap\cO_C\simeq f_{C*}\cO_{C_Y}/f_{C*}\fn_a$. For $a$ in Lemma \ref{lem:topology}, it is $\cO_C/f_{C*}\fn_a\simeq f_{C*}\cO_{C_Y}/f_{C*}\fn_a$. Therefore $\cO_C\simeq f_{C*}\cO_{C_Y}$, meaning the normality of $C$.
\end{proof}

Theorem \ref{thm:centre} is established.

\begin{remark}
\begin{enumerate}
\item
One may prove the normality of $C$ by using Zariski's subspace theorem \cite[(10.6)]{Ab98}. One has an isomorphism $\cO_C/f_{C*}\fn_a\cap\cO_C\simeq f_{C*}(\cO_{C_Y}/\fn_a)$ for any $a$. By (\ref{eqn:fibre}), the family $\{\fn_a\}_a$ gives the $\fn\cO_{C_Y}$-adic topology. Since the family $\{f_*\cO_Y(-mA)\}_m$ in the proof of Lemma \ref{lem:topology} gives the $\fm$-adic topology by Zariski's subspace theorem (cf.\ \cite[Lemma 3]{K07}), we see from (\ref{eqn:Serre}) that the family $\{f_{C*}\fn_a\cap\cO_C\}_a$ as well as $\{f_{C*}\fn^a\cO_{C_Y}\cap\cO_C\}_a$ gives the $\fm\cO_C$-adic topology. Hence $\widehat{\cO_{C,P}}\simeq\varprojlim_a\cO_C/f_{C*}\fn_a\cap\cO_C\simeq\varprojlim_af_{C*}(\cO_{C_Y}/\fn_a)\simeq\widehat{\cO_{C_Y,P_Y}}$ and $C$ is normal by \cite[Proposition IV.2.1.13]{EGA}.
\item
The author used Zariski's subspace theorem in the proof of \cite[(10)]{K13-1}, but it derives only the inclusion $\bar\varphi_*\cO_{\bar{X}}(-l_2E_Z)\subset\cI_Z^{(l_1)}$ for the $l$-th symbolic power $\cI_Z^{(l)}$ of $\cI_Z$. In order to obtain \cite[(10)]{K13-1}, we need the equivalence of the $\cI_Z$-adic topology and the $\cI_Z$-symbolic topology by \cite[\S6 Lemma 3]{Z51} (see also \cite{Sc85}, \cite{Ve88}).
\end{enumerate}
\end{remark}

\section{The ACC for minimal log discrepancies}\label{sec:acc}
In this section, we discuss the ACC for minimal log discrepancies on non-singular varieties from the point of view of generic limits.

\subsection{Statements}
We begin with the statement of the ACC conjecture.

\begin{definition}
We say that a subset $I$ of $\bR$ satisfies the \textit{ascending chain condition} (\textit{ACC}) (resp.\ the \textit{descending chain condition} (\textit{DCC})) if there exist no infinite strictly increasing (resp.\ strictly decreasing) sequences of elements in $I$.
\end{definition}

\begin{remark}\label{rmk:cc}
$I\subset\bR$ is finite if and only if $I$ satisfies both the ACC and DCC.
\end{remark}

\begin{definition}
Let $P\in(X,\Delta=\sum_i\delta_i\Delta_i,\fa=\prod_j\fa_j^{r_j})$ be a germ of a triplet. We write $\Coef_P(\Delta,\fa)$ for the set which consists of all $\delta_i>0$ with $\Delta_i$ passing through $P$ and all $r_j>0$ with $\fa_j$ non-trivial at $P$.
\end{definition}

\begin{conjecture}[Shokurov \cite{Sh88}, \cite{Sh96}, McKernan \cite{M-MSRI}]\label{cnj:acc}
Fix $d\in\bN$ and subsets $I\subset(0,\infty)$ and $J\subset[0,\infty)$ both of which satisfy the DCC. Then there exist finite subsets $I_0\subset I$ and $J_0\subset J$ such that if $P\in(X,\Delta,\fa)$ is a germ of a triplet on a variety $X$ of dimension $d$ with $\Coef_P(\Delta,\fa)\subset I$ and $\mld_P(X,\Delta,\fa)\in J$, then $\Coef_P(\Delta,\fa)\subset I_0$ and $\mld_P(X,\Delta,\fa)\in J_0$.
\end{conjecture}

Conjecture \ref{cnj:acc} by McKernan is a generalisation of the original conjecture by Shokurov, which claims only the existence of $J_0$. When $d=2$, the existence of $J_0$ was proved by Alexeev \cite{Al93}. The motivation of this conjecture stems from the reduction by Shokurov \cite{Sh04} that the termination of flips follows from two conjectural properties of minimal log discrepancies: the ACC and the lower semi-continuity. For the purpose of the termination of flips, one may assume $I$ in Conjecture \ref{cnj:acc} to be a finite set.

We consider Conjecture \ref{cnj:acc} with the assumption of the non-singularity of $X$. Then we may assume $\Delta=0$ by absorbing $\Delta$ to $\fa$, since any divisor on $X$ is a Cartier divisor.

\begin{conjecture'acc}
Fix $d\in\bN$ and subsets $I\subset(0,\infty)$ and $J\subset[0,d]$ both of which satisfy the DCC. Then there exist finite subsets $I_0\subset I$ and $J_0\subset J$ such that if $P\in(X,\fa)$ is a germ of a pair on a non-singular variety $X$ of dimension $d$ with $\Coef_P\fa\subset I$ and $\mld_P(X,\fa)\in J$, then $\Coef_P\fa\subset I_0$ and $\mld_P(X,\fa)\in J_0$.
\end{conjecture'acc}

Theorem \ref{thm:acc} is Conjecture \ref{cnj:acc}$\mathrm{'}$ for $d=3$ with $J\subset(1,3]$. Conjecture \ref{cnj:acc}$\mathrm{'}$ with $I$ finite was proved in \cite{K14}.

\subsection{Reduction}
We shall reduce Conjecture \ref{cnj:acc}$\mathrm{'}$ to the stability of minimal log discrepancies in taking a generic limit of $\bR$-ideals. We refer to Appendix \ref{sec:appendix} for the definition of a generic limit and the relevant notation: $R=k[[x_1,\ldots,x_d]]$ with maximal ideal $\fm$ and $X=\Spec R$ with closed point $P$, and for a field extension $K$ of $k$, $R_K=K[[x_1,\ldots,x_d]]$ with maximal ideal $\fm_K$ and $X_K=\Spec R_K$ with closed point $P_K$.

\begin{conjecture}[{\cite[Conjecture 5.7]{K14}}]\label{cnj:stability}
Fix $r_1,\ldots,r_e>0$. Let $S=\{(\fa_{i1},\ldots,\fa_{ie})\}_{i\in I}$ be a collection of $e$-tuples of ideals in $R =k[[x_1,\ldots,x_d]]$, and $(\sfa_1,\ldots,\sfa_e)$ the generic limit of $S$ defined in $R_K$ with respect to a family $\cF$ of approximations of $S$. Set $\fa_i=\prod_j\fa_{ij}^{r_j}$ and $\sfa=\prod_j\sfa_j^{r_j}$. Then after replacing $\cF$ with a subfamily,
\begin{align*}
\mld_{P_K}(X_K,\sfa)=\mld_P(X,\fa_i)
\end{align*}
for any $i\in I_l$.
\end{conjecture}

Conjecture \ref{cnj:stability} is closely related to the ideal-adic semi-continuity of minimal log discrepancies.

\begin{conjecture}[Musta\c{t}\u{a}, cf.\ {\cite[Conjecture 2.5]{K13-1}}]\label{cnj:adic}
Let $P\in X=\Spec k[[x_1,\ldots,x_d]]$ and $\fm$ be as above and $\fa=\prod_j\fa_j^{r_j}$ an $\bR$-ideal on $X$. Then there exists an integer $l$ such that if an $\bR$-ideal $\fb=\prod_j\fb_j^{r_j}$ on $X$ satisfies $\fa_j+\fm^l=\fb_j+\fm^l$ for all $j$, then $\mld_P(X,\fa)=\mld_P(X,\fb)$.
\end{conjecture}

\begin{remark}\label{rmk:inequality}
One inequality is easy in both conjectures. One has $\mld_{P_K}(X_K,\sfa)\ge\mld_P(X,\fa_i)$ in Conjecture \ref{cnj:stability} by Lemma \ref{lem:mld}, and $\mld_P(X,\fa)\ge\mld_P(X,\fb)$ in Conjecture \ref{cnj:adic} by \cite[Remark 2.5.3]{K13-1}. In particular, these conjectures hold in the case when $(X_K,\sfa)$ (resp.\ $(X,\fa)$) is not lc.
\end{remark}

\begin{proposition}\label{prp:reduction}
Conjecture \textup{\ref{cnj:stability}} implies Conjectures \textup{\ref{cnj:acc}$\mathrm{'}$} and \textup{\ref{cnj:adic}}.
\end{proposition}

\begin{proof}
Firstly, we shall see Conjecture \ref{cnj:acc}$\mathrm{'}$. It was observed by Musta\c{t}\u{a} and sketched in \cite[Remark 2.5.1]{K13-1}. Let $\{\fa_i=\prod_{j=1}^{e_i}\fa_{ij}^{r_{ij}}\}_{i\in N}$ be an arbitrary collection of $\bR$-ideals on $X=\Spec R$ such that $\fa_{ij}$ are non-trivial at $P$, $r_{ij}\in I$ and $m_i:=\mld_P(X,\fa_i)\in J$. Then $\sum_{j=1}^{e_i}r_{ij}\le\ord_E\fa_i\le a_E(X)=d$ for the divisor $E$ obtained by the blow-up of $X$ at $P$, since $m_i\ge0$. The $I$ has the minimum, say $\iota>0$, so $e_i\le\iota^{-1}d$. By Corollary \ref{cor:extension} and Remark \ref{rmk:cc}, it is enough to show that both the subsets $\bigcup_{i\in N}\Coef_P\fa_i$ of $I$ and $\bigcup_{i\in N}\{m_i\}$ of $J$ satisfy the ACC. We may replace $N$ with a countable subset $\bN$ on which $e_i$ is constant, say $e$, such that the sequences $\{r_{ij}\}_{i\in\bN}$ for $1\le j\le e$ and $\{m_i\}_{i\in\bN}$ are non-decreasing. By $r_{ij}\le d$ and $m_i\le d$, these sequences have limits $r_j:=\lim_ir_{ij}$ and $m:=\lim_im_i$. It suffices to prove $r_{ij}=r_j$ and $m_i=m$ for some $i$.

For the collection $S=\{(\fa_{i1},\ldots,\fa_{ie})\}_{i\in\bN}$ of $e$-tuples of ideals in $R$, we take a family $\cF=(Z_l,(\bar\sfa_j(l))_j,I_l,s_l,t_{l+1})_{l\ge l_0}$ of approximations of $S$ and the generic limit $(\sfa_1,\ldots,\sfa_e)$ of $S$ defined in $R_K$ with respect to $\cF$ as in Lemma \ref{lem:mld}, where $E_K\in\cD_{X_K}$ computing $M:=\mld_{P_K}(X_K,\prod_j\sfa_j^{r_j})$ is fixed. It is extended to $E_l$ over $X\times_{\Spec k}Z_l$, and for $i\in I_l$ with $z=s_l(i)$ we have $M=\mld_P(X,\prod_j(\fa_{ij}+\fm^l)^{r_j})=a_{(E_l)_z}(X,\prod_j(\fa_{ij}+\fm^l)^{r_j})$ and $\ord_{E_K}\sfa_j=\ord_{(E_l)_z}(\fa_{ij}+\fm^l)<l$ using (\ref{itm:family_aij}) in Definition \ref{def:family}. Hence $\ord_{E_K}\sfa_j=\ord_{(E_l)_z}\fa_{ij}$ and
\begin{align}\label{eqn:mM}
m_i\le a_{(E_l)_z}(X,\prod_j\fa_{ij}^{r_{ij}})&=a_{(E_l)_z}(X,\prod_j\fa_{ij}^{r_j})+\sum_j(r_j-r_{ij})\ord_{(E_l)_z}\fa_{ij}\\
\nonumber&=M+\sum_j(r_j-r_{ij})\ord_{E_K}\sfa_j.
\end{align}

By Conjecture \ref{cnj:stability}, $M=\mld_P(X,\prod_j\fa_{ij}^{r_j})\le m_i$ for any $i\in I_l$ after replacing $\cF$ with a subfamily. With (\ref{eqn:mM}), we obtain
\begin{align*}
M\le m_i\le M+\sum_j(r_j-r_{ij})\ord_{E_K}\sfa_j.
\end{align*}
The right-hand side converges to $M$, whence $m_i=m=M$. Then $\mld_P(X,\prod_j\fa_{ij}^{r_{ij}})=\mld_P(X,\prod_j\fa_{ij}^{r_j})$, so $r_{ij}=r_j$.

Secondly, we shall see Conjecture \ref{cnj:adic}. Suppose the contrary. Then for every $i\in\bN$, there exists an $\bR$-ideal $\fb_i=\prod_j\fb_{ij}^{r_j}$ on $X$ such that $\fa_j+\fm^i=\fb_{ij}+\fm^i$ for all $j$ but $\mld_P(X,\fa)\ne\mld_P(X,\fb_i)$. Take a family $\cF=(Z_l,(\bar\sfb_j(l))_j,I_l,s_l,t_{l+1})_{l\ge l_0}$ of approximations of $S=\{(\fb_{ij})_j\}_{i\in\bN}$ and the generic limit $(\sfb_j)_j$ of $S$ defined in $R_K$ with respect to $\cF$. Then for $l\ge l_0$, $\bar\sfb_j(l)_zR=\fb_{ij}+\fm^l=\fa_j+\fm^l$ for $i\in I_l$ with $z=s_l(i)$ satisfying $i\ge l$, and such $z$ form a dense subset of $Z_l$. This implies $\bar\sfb_j(l)=((\fa_j+\fm^l)\cap\bar{R})\otimes_k\cO_{Z_l}$, whence $\bar\sfb_j(l)_K=(\fa_jR_K+\fm_K^l)\cap\bar{R}_K$. Then $\sfb_j=\varprojlim_l\bar\sfb_j(l)_K=\fa_jR_K$ by Remark \ref{rmk:limit}, so $\mld_{P_K}(X_K,\prod_j\sfb_j^{r_j})=\mld_P(X,\fa)$ by Corollary \ref{cor:extension}. By Conjecture \ref{cnj:stability}, we have $\mld_{P_K}(X_K,\prod_j\sfb_j^{r_j})=\mld_P(X,\fb_i)$ for infinitely many $i$, that is, $\mld_P(X,\fa)=\mld_P(X,\fb_i)$, which is absurd.
\end{proof}

\begin{remark}\label{rmk:reduction}
Proposition \ref{prp:reduction} has the refinement that for fixed $d$ and $a\ge0$,
\begin{enumerate}
\item
Conjecture \textup{\ref{cnj:stability}} for $d$ with $\mld_{P_K}(X_K,\sfa)>a$ (resp.\ $\ge a$) implies Conjecture \textup{\ref{cnj:acc}$\mathrm{'}$} for $d$ with $J\subset(a,d]$ (resp.\ $\subset[a,d]$), and
\item
Conjecture \textup{\ref{cnj:stability}} for $d$ with $\mld_{P_K}(X_K,\sfa)=a$ implies Conjecture \textup{\ref{cnj:adic}} for $d$ with $\mld_P(X,\fa)=a$.
\end{enumerate}
This is obvious by the above proof. Note that (\ref{eqn:mM}) implies $m\le M$.
\end{remark}

\begin{remark}
Theorem \ref{thm:adic0} gives Conjecture \ref{cnj:adic} in the case when $\mld_P(X,\fa)=0$, and then its Corollary \ref{cor:adic0} gives Conjecture \ref{cnj:stability} in the case when $\mld_{P_K}(X_K,\sfa)=0$. The order of this logic is opposite to Proposition \ref{prp:reduction}. We expect that an effective estimate of $l$ in Conjecture \ref{cnj:adic} implies Conjecture \ref{cnj:stability}.
\end{remark}

Theorem \ref{thm:adic0} is reduced to the corresponding statement \cite[Theorem 1.4]{dFEM10} on a variety by the property that the log canonical threshold for an ideal in $\widehat{\cO_{Y,Q}}$ is approximated by those for ideals in $\cO_{Y,Q}$. This property for the minimal log discrepancy on $X$ is a special case of Conjecture \ref{cnj:stability}, so we do not know how to reduce Conjecture \ref{cnj:adic} to its variety version. The version of Conjecture \ref{cnj:adic} for a germ $Q\in(Y,\Delta,\fa)$ of a triplet on a variety $Y$ holds when (i) $(Y,\Delta,\fa)$ is klt \cite[Theorem 2.6]{K13-1}, (ii) $Y$ is a surface \cite{K13-2}, or (iii) $Y$ is toric and $Q$, $\Delta$, $\fa$ are torus invariant \cite[Theorem 1.8]{N-Arxiv}.

The variety version of Theorem \ref{thm:adic0} is globalised.

\begin{theorem}\label{thm:global}
Let $(Y,\Delta,\fa=\prod_j\fa_j^{r_j})$ be a triplet on a variety $Y$ and $Z$ an irreducible closed subset of $Y$. Suppose $\mld_{\eta_Z}(Y,\Delta,\fa)=0$ and it is computed by $E\in\cD_Y$. Then there exists an open subset $Y'$ of $Y$ containing $\eta_Z$ such that if an $\bR$-ideal $\fb=\prod_j\fb_j^{r_j}$ on $Y'$ satisfies $\fa_j|_{Y'}+\fp_j=\fb_j+\fp_j$ for all $j$, where $\fp_j=\{u\in\cO_{Y'}\mid\ord_Eu>\ord_E\fa_j\}$, then $(Y',\Delta|_{Y'},\fb)$ is lc about $Z|_{Y'}$ and $\mld_{\eta_Z}(Y',\Delta|_{Y'},\fb)=0$.
\end{theorem}

\begin{proof}
Take a log resolution $f\colon W\to Y$ of $(Y,\Delta,\fa\cI_Z)$, where $\cI_Z$ is the ideal sheaf of $Z$, such that $E$ is realised as a divisor on $W$. Then $F:=\Exc f\cup\Supp\Delta_W\cup\Cosupp\fa\cI_Z\cO_Y$ is an snc divisor $\sum_iF_i$, where $\Delta_W$ is the strict transform of $\Delta$. By generic smoothness \cite[Corollary III.10.7]{Ha77}, there exists an open subset $Y'$ of $Y$ containing $\xi_Z$ such that if the restriction $S'=S|_{f^{-1}(Y')}$ of a stratum $S$ of $\sum_iF_i$ satisfies $S'\neq\emptyset$ and $f(S')\subset Z'=Z|_{Y'}$, then $S'\to Z'$ is smooth and surjective. Then for any $Q\in Z'$, $\mld_Q(Y,\Delta,\fa\fm_Q^{\dim Z})=0$ for the maximal ideal sheaf $\fm_Q$, and it is computed by the divisor $G_Q$ obtained by the blow-up of $W$ along a component of $E\cap f^{-1}(Q)$. Since $\ord_{G_Q}\fa_j=\ord_E\fa_j$ and $\ord_{G_Q}u\ge\ord_Eu$ for $u\in\cO_{Y'}$, we have $\mld_Q(Y',\Delta|_{Y'},\fb\fm_Q^{\dim Z})=0$ for $\fb$ in Theorem \ref{thm:global} by \cite[Theorem 1.4]{dFEM10} (its proof works for triplets). Hence $(Y',\Delta|_{Y'},\fb)$ is lc about $Z'$, and $\mld_{\eta_Z}(Y',\Delta|_{Y'},\fb)=0$ by $a_E(Y',\Delta|_{Y'},\fb)=0$.
\end{proof}

\begin{corollary}
Let $(Y,\Delta,\fa=\prod_j\fa_j^{r_j})$ be an lc triplet on a variety $Y$ and $Z$ a closed subset of $Y$ with ideal sheaf $\cI_Z$. Then there exists an integer $l$ such that if an $\bR$-ideal $\fb=\prod_j\fb_j^{r_j}$ on $Y$ satisfies $\fa_j+\cI_Z^l=\fb_j+\cI_Z^l$ for all $j$, then $(Y,\Delta,\fb)$ is lc about $Z$.
\end{corollary}

\begin{remark}
The author should have written the proof after \cite[Theorem 2.4]{K13-1}. The estimate of $l$ in \cite[Remark 2.4.1]{K13-1} is an error unless $Z$ is a closed point (so is that of $l_1$ in \cite[Lemma 3.1]{K13-1}).
\end{remark}

\section{The klt and plt cases}\label{sec:kltplt}
In this section, we settle Conjecture \ref{cnj:stability} in the klt case, and in the plt case whose lc centre has an isolated singularity. We keep the notation in Appendix \ref{sec:appendix}, so $P\in X=\Spec R$ with $R=k[[x_1,\ldots,x_d]]$ and $P_K\in X_K=\Spec R_K$ with $R_K=K[[x_1,\ldots,x_d]]$.

\subsection{The klt case}
\begin{theorem}\label{thm:klt}
Conjecture \textup{\ref{cnj:stability}} holds in the case when $(X_K,\sfa)$ is klt.
\end{theorem}

\begin{proof}
It is shown similarly to \cite[Theorem 2.6]{K13-1}. By Remark \ref{rmk:inequality}, it suffices to show that after replacing $\cF$ with a subfamily,
\begin{align}\label{eqn:klt_inequality}
a_E(X,\fa_i)\ge\mld_{P_K}(X_K,\sfa)
\end{align}
for any $i\in I_l$ and $E\in\cD_X$ with centre $P$.

Take a subfamily in Lemma \ref{lem:mld} so that $\mld_P(X,\prod_j(\bar\sfa_j(l)_zR)^{r_j})=\mld_{P_K}(X_K,\sfa)$ for $z\in Z_l$. Then for $i\in I_l$,
\begin{align}\label{eqn:klt_approx}
a_E(X,\prod_j(\fa_{ij}+\fm^l)^{r_j})\ge\mld_{P_K}(X_K,\sfa)
\end{align}
by (\ref{itm:family_aij}) in Definition \ref{def:family}. Since $(X_K,\sfa)$ is klt, we can fix $t>0$ such that $(X_K,\sfa^{1+t})$ is lc. By Corollary \ref{cor:adic0}, $(X,\fa_i^{1+t})$ is lc for $i\in I_l$ after replacing $\cF$ with a subfamily, whence $a_E(X,\fa_i)\ge t\ord_E\fa_i=t\sum_jr_j\ord_E\fa_{ij}$. We fix $l\ge l_0$ such that $l\ge(tr_j)^{-1}\mld_{P_K}(X_K,\sfa)$ for all $j$. Then,
\begin{align}\label{eqn:klt_klt}
a_E(X,\fa_i)\ge l^{-1}\ord_E\fa_{ij}\cdot\mld_{P_K}(X_K,\sfa)
\end{align}
for any $j$ and $i\in I_l$.

If $\ord_E\fa_{ij}<l$ for all $j$, then $\ord_E\fa_{ij}=\ord_E(\fa_{ij}+\fm^l)$, so one has $a_E(X,\fa_i)=a_E(X,\prod_j(\fa_{ij}+\fm^l)^{r_j})$, and (\ref{eqn:klt_inequality}) follows from (\ref{eqn:klt_approx}). If $\ord_E\fa_{ij}\ge l$ for some $j$, then (\ref{eqn:klt_inequality}) follows from (\ref{eqn:klt_klt}).
\end{proof}

\begin{remark}\label{rmk:remain}
By Remark \ref{rmk:inequality}, Theorem \ref{thm:klt} and Corollary \ref{cor:adic0}, Conjecture \ref{cnj:stability} remains open only when $(X_K,\sfa)$ is non-klt with $\mld_{P_K}(X_K,\sfa)>0$.
\end{remark}

\subsection{The plt case whose lc centre has an isolated singularity}
Suppose that $(X_K,\sfa)$ is an lc but not klt pair every lc centre of which has codimension $1$. Then by Proposition \ref{prp:centre}, $(X_K,\sfa)$ has the smallest lc centre $S_K$ and it is normal. We prove Conjecture \ref{cnj:stability} on the assumption that $S_K$ has an isolated singularity.

\begin{theorem}\label{thm:plt}
Conjecture \textup{\ref{cnj:stability}} holds in the case when $(X_K,\sfa)$ has the smallest lc centre of codimension $1$ which is non-singular outside $P_K$.
\end{theorem}

We let $S_K$ denote the smallest lc centre of $(X_K,\sfa)$. $S_K$ is a prime divisor which is non-singular outside $P_K$. We define an $\bR$-ideal $\sfc=\prod_j\sfc_j^{r_j}$ by the expression $\sfa_j=\sfc_j\cO_{X_K}(-m_jS_K)$ with $\sum_jr_jm_j=1$. The $\sfa$ and $\sfc\cO_{X_K}(-S_K)$ take the same order along any divisor over $X_K$. We can fix $t>0$ such that $(X_K,S_K,\sfc^{1+t})$ is lc, since $S_K$ is the unique lc centre of $(X_K,S_K,\sfc)$.

We take a log resolution $f_K\colon Y_K\to X_K$ of $(X_K,S_K,\fm_K)$, which is isomorphic outside $P_K$. Let $\{E_{\alpha K}\}_\alpha$ be the set of all $f_K$-exceptional prime divisors. The $E_K=\sum_\alpha E_{\alpha K}$ is snc. Let $(Y_K,\Delta_K,\sfa'=\prod_j(\sfa'_j)^{r_j})$ be the pull-back of $(X_K,0,\sfa)$ and $(Y_K,T_K+\Delta_K,\sfc')$ that of $(X_K,S_K,\sfc)$. We set
\begin{align*}
L_K&:=T_K\cap f_K^{-1}(P_K),\\
C_K&:=\Cosupp\sfc'\cap f_K^{-1}(P_K).
\end{align*}
By blowing up $Y_K$ further, we may assume that $C_K$ is contained in the union of those $E_{\alpha K}$ satisfying
\begin{align}\label{eqn:tcm}
t\ord_{E_{\alpha K}}\sfc\ge\mld_{P_K}(X_K,\sfa).
\end{align}
One sees this by induction on $\displaystyle\max_J\{\min_{\alpha\in J}\{\ord_{E_{\alpha K}}\sfc\}\}$ in which one considers all subsets $J$ of indices satisfying $C_K\subset\bigcup_{\alpha\in J}E_{\alpha K}$, since the order of $\sfc$ takes value in the discrete subset $\sum_jr_j\bZ_{\ge0}$ of $\bR$.

The $f_K$ is descendible by Proposition \ref{prp:descendible}, so replacing $\cF$ with a subfamily, we obtain the diagram (\ref{eqn:cartesian}) in which $\bar{f}_l$ is a family of log resolutions. Shrinking $Z_l$, we may assume that $E_{\alpha K}$, $L_K$ and $C_K$ are the base changes of flat families $\bar{E}_{\alpha l}$, $\bar{L}_l$ and $\bar{C}_l$ in $\bar{Y}_l$ over $Z_l$. We may assume that $\sum_\alpha\bar{E}_{\alpha l}$ is an snc divisor, that the projections to $Z_l$ from every stratum of $\sum_\alpha\bar{E}_{\alpha l}$ and from its intersection with $\bar{L}_l$ are smooth and surjective, and that $\ord_{(\bar{E}_{\alpha l})_z}\bar\sfa_j(l)_z$ is constant on $z\in Z_l$ for each $\alpha$ and $j$. Their base changes in $Y_l$ are denoted by $E_{\alpha l}$, $L_l$ and $C_l$. We write $\bar{E}_l=\sum_\alpha\bar{E}_{\alpha l}$ and $E_l=\sum_\alpha E_{\alpha l}$.

We fix $m$ such that $m\ord_{E_{\alpha K}}\fm_K\ge\ord_{E_{\alpha K}}\sfc_j$ for all $\alpha$ and $j$, and set
\begin{align*}
\sfd:=\prod_j(\sfc_j+\fm_K^{m})^{tr_j}.
\end{align*}
Then $\ord_{E_{\alpha K}}\sfd=t\ord_{E_{\alpha K}}\sfc$, and $(X_K,\sfa\sfd)$ is lc. The $\sfd$ is defined over some $k(Z_l)$, so by replacing $\cF$ with a subfamily, we may assume that $\sfd$ is the base change of an $\bR$-ideal $\bar\fd_l=\prod_j\bar\fd_{lj}^{tr_j}$ on $\bA_k^d\times_{\Spec k}Z_l$ with $\bar\fm^m\otimes_k\cO_{Z_l}\subset\bar\fd_{lj}$ and that $\ord_{(\bar{E}_{\alpha l})_z}(\bar\fd_l)_z$ is constant on $Z_l$ for each $\alpha$. By Corollary \ref{cor:adic0}, after taking a subfamily, $(X,\fa_i(\fd_l)_z)$ is lc for any $i\in I_l$ with $z=s_l(i)$, where $\fd_l$ is the pull-back on $X\times_{\Spec k}Z_l$ of $\bar\fd_l$.

We fix $l\ge l_0$ such that
\begin{align}\label{eqn:lma}
l\ord_{E_{\alpha K}}\fm_K>\ord_{E_{\alpha K}}\sfa_j+\ord_{S_K}\sfa_j
\end{align}
for all $\alpha$ and $j$. By Remark \ref{rmk:inequality}, for Theorem \ref{thm:plt} it suffices to prove that after shrinking $Z_l$,
\begin{align}\label{eqn:plt_inequality}
a_E(X,\fa_i)\ge\mld_{P_K}(X_K,\sfa)
\end{align}
for any $i\in I_l$ and $E\in\cD_X$ with centre $P$. Setting $z=s_l(i)$, we shall prove (\ref{eqn:plt_inequality}) by treating the three cases according to the position of $c_{(Y_l)_z}(E)$:
\begin{enumerate}
\item[(a)]
$c_{(Y_l)_z}(E)\not\subset(L_l\cup C_l)_z$.
\item[(b)]
$c_{(Y_l)_z}(E)\subset(C_l)_z$.
\item[(c)]
$c_{(Y_l)_z}(E)\subset(L_l)_z$ and $c_{(Y_l)_z}(E)\not\subset(C_l)_z$.
\end{enumerate}

We let $\bar\sfa'(l)=\prod_j\bar\sfa'_j(l)^{r_j}$ be the weak transform on $\bar{Y}_l$ of $\prod_j\bar\sfa_j(l)^{r_j}$, and $\fa'_i=\prod_j(\fa'_{ij})^{r_j}$ the weak transform on $(Y_l)_z$ of $\fa_i$.

\begin{lemma}\label{lem:weaktf}
\begin{enumerate}
\item\label{itm:weaktf_K}
$\bar\sfa'_j(l)_K\cO_{Y_K}=\sfa'_j+\sfI_{lj}$ with an ideal sheaf $\sfI_{lj}$ which is contained in $\cO_{Y_K}(-E_K)^{\ord_{S_K}\sfa_j+1}$.
\item\label{itm:weaktf_z}
$\bar\sfa'_j(l)_z\cO_{(Y_l)_z}=\fa'_{ij}+\cI_{lij}$ with an ideal sheaf $\cI_{lij}$ which is contained in $\cO_{(Y_l)_z}(-(E_l)_z)^{\ord_{S_K}\sfa_j+1}$.
\item\label{itm:weaktf_cosupp}
$\Cosupp\bar\sfa'(l)=\bar{L}_l\cup\bar{C}_l$ after shrinking $Z_l$.
\end{enumerate}
\end{lemma}

\begin{proof}
Write $\fm_K\cO_{Y_K}=\cO_{Y_K}(-M_K)$ and $\sfa_j\cO_{Y_K}=\sfa'_j\cO_{Y_K}(-A_{jK})$. The inequality (\ref{eqn:lma}) means that $\sfI_{lj}=\cO_{Y_K}(A_{jK}-lM_K)$ is an ideal sheaf contained in $\cO_{Y_K}(-E_K)^{\ord_{S_K}\sfa_j+1}$. Then $\fm_K^l\cO_{Y_K}=\sfI_{lj}\cO_{Y_K}(-A_{jK})$, which induces (\ref{itm:weaktf_K}) by $\bar\sfa_j(l)_KR_K=\sfa_j+\fm_K^l$. From (\ref{itm:weaktf_K}), $\Cosupp(\bar\sfa'(l)_K\cO_{Y_K})=\Cosupp\sfa'\cap f_K^{-1}(P_K)=L_K\cup C_K$, which is extended to $\Cosupp\bar\sfa'(l)=\bar{L}_l\cup\bar{C}_l$ in (\ref{itm:weaktf_cosupp}). On the other hand, $\ord_{E_{\alpha K}}\fm_K=\ord_{(E_{\alpha l})_z}\fm$ and $\ord_{E_{\alpha K}}\sfa_j=\ord_{E_{\alpha K}}\bar\sfa_j(l)_KR_K=\ord_{(\bar{E}_{\alpha l})_z}\bar\sfa_j(l)_z=\ord_{(E_{\alpha l})_z}\fa_{ij}$ by (\ref{eqn:lma}) and Definitions \ref{def:family}, \ref{def:limit}. Then, (\ref{itm:weaktf_z}) is induced similarly to (\ref{itm:weaktf_K}).
\end{proof}

The cases (a) and (b) are not difficult.

\begin{proof}[Proof of \textup{(\ref{eqn:plt_inequality})} in the case \textup{(a)}]
Set $\Delta_l=\sum_\alpha(1-a_{E_{\alpha K}}(X_K,\sfa))E_{\alpha l}$, base-changed to $\Delta_K$. Then $((Y_l)_z,(\Delta_l)_z,\fa'_i)$ is the pull-back of $(X,0,\fa_i)$. We have $a_E((Y_l)_z,(\Delta_l)_z)\ge\ord_E(E_l-\Delta_l)_z$ by the log canonicity of $((Y_l)_z,(E_l)_z)$. For a divisor $(E_{\alpha l})_z$ containing $c_{(Y_l)_z}(E)$, we have $\ord_E(E_l-\Delta_l)_z\ge\ord_{(E_{\alpha l})_z}(E_l-\Delta_l)_z=a_{E_{\alpha K}}(X_K,\sfa)\ge\mld_{P_K}(X_K,\sfa)$. Hence $a_E((Y_l)_z,(\Delta_l)_z)\ge\mld_{P_K}(X_K,\sfa)$. By Lemma \ref{lem:weaktf}(\ref{itm:weaktf_z}) and (\ref{itm:weaktf_cosupp}), $\Cosupp\fa'_i\cap(f_l)_z^{-1}(P)=\Cosupp\bar\sfa'(l)_z\cO_{(Y_l)_z}=(L_l\cup C_l)_z$, so $\ord_E\fa'_i=0$. Thus $a_E(X,\fa_i)=a_E((Y_l)_z,(\Delta_l)_z,\fa'_i)=a_E((Y_l)_z,(\Delta_l)_z)\ge\mld_{P_K}(X_K,\sfa)$.
\end{proof}

\begin{proof}[Proof of \textup{(\ref{eqn:plt_inequality})} in the case \textup{(b)}]
The $c_{(Y_l)_z}(E)$ lies on some $(E_{\alpha l})_z$ such that $E_{\alpha K}$ satisfies (\ref{eqn:tcm}). Then $\ord_E(\fd_l)_z\ge\ord_{(E_{\alpha l})_z}(\fd_l)_z=\ord_{E_{\alpha K}}\sfd=t\ord_{E_{\alpha K}}\sfc\ge\mld_{P_K}(X_K,\sfa)$. On the other hand, the log canonicity of $(X,\fa_i(\fd_l)_z)$ implies $a_E(X,\fa_i)\ge\ord_E(\fd_l)_z$. These two inequalities are joined as $a_E(X,\fa_i)\ge\mld_{P_K}(X_K,\sfa)$.
\end{proof}

The case (c) is reduced to the following log canonicity.

\begin{lemma}\label{lem:notC}
After shrinking $Z_l$, the triplet $((Y_l)_z,(E_l)_z,\fa'_i)$ is lc about $(L_l)_z\setminus(C_l)_z$ for any $i\in I_l$ with $z=s_l(i)$.
\end{lemma}

\begin{proof}[Proof of \textup{(\ref{eqn:plt_inequality})} in the case \textup{(c)} from Lemma \textup{\ref{lem:notC}}]
For $\Delta_l=\sum_\alpha(1-a_{E_{\alpha K}}(X_K,\sfa))E_{\alpha l}$, we have $a_E(X,\fa_i)=a_E((Y_l)_z,(\Delta_l)_z,\fa'_i)\ge\ord_E(E_l-\Delta_l)_z$ by Lemma \ref{lem:notC}, and have seen $\ord_E(E_l-\Delta_l)_z\ge\mld_{P_K}(X_K,\sfa)$ in the proof in the case (a).
\end{proof}

\begin{proof}[Proof of Lemma \textup{\ref{lem:notC}}]
Pick any open stratum $F_K$ of the snc divisor $E_K$, which is extended to an open stratum $\bar{F}_l$ of $\bar{E}_l$. We prove Lemma \ref{lem:notC} by noetherian induction. Recall that $l$ has been fixed. Let $\bar{Q}_l$ be an irreducible locally closed subset of $\bar{F}_l\cap\bar{L}_l\setminus\bar{C}_l$ which dominates $Z_l$. It suffices to show that the existence of a dense open subset $\bar{Q}_l^\circ$ of $\bar{Q}_l$ such that the triplet $((Y_l)_z,(E_l)_z,\fa'_i)$ is lc about $(Q_l^\circ)_z$ for $i\in I_l$ with $z=s_l(i)$, where $Q_l^\circ=\bar{Q}_l^\circ\times_{\bar{Y}_l}Y_l$.

By shrinking $\bar{Q}_l$ and $Z_l$, we may assume that $\bar{Q}_l\to Z_l$ is smooth and surjective. Let $\bar{f}_l^+\colon\bar{Y}_l^+\to\bA_k^d\times_{\Spec k}\bar{Q}_l$ be the base change of $\bar{f}_l$ by $\bar{Q}_l\to Z_l$. Then $\pr_{\bar{Q}_l}\circ\bar{f}_l^+$ has the natural section $\bar{g}_l\colon\bar{Q}_l\to\bar{Y}_l^+=\bar{Y}_l\times_{Z_l}\bar{Q}_l$ by the immersion $\bar{Q}_l\hookrightarrow\bar{Y}_l$. We construct $f_K^+$ and $g_K$ similarly for $Q_K=\bar{Q}_l\times_{\bar{Y}_l}Y_K$ as below.
\begin{align*}
\xymatrix{
Y_K^+\ar[r]\ar[d]_{f_K^+}&\bar{Y}_l^+\ar[r]\ar[d]_{\bar{f}_l^+}&\bar{Y}_l\ar[d]_{\bar{f}_l}\\
X_K\times_{\Spec K}Q_K\ar[];+/r16mm/\ar[d]&\bA_k^d\times_{\Spec k}\bar{Q}_l\ar[r]\ar[d]&\bA_k^d\times_{\Spec k}Z_l\ar[d]\\
Q_K\ar[r]\ar@/^3pc/[uu]^(.7){g_K}&\bar{Q}_l\ar[r]\ar@/^3pc/[uu]^(.7){\bar{g}_l}&Z_l
}
\end{align*}

The $\bar{Y}_l^+$, $Y_K^+$ are the base changes of $\bar{Y}_l$, $Y_K$ by smooth morphisms. For a staff $\square$ on $\bar{Y}_l$ or $Y_K$, we mean by $\square^+$ the base change of $\square$ on $\bar{Y}_l^+$ or $Y_K^+$. For example, $\sfa'^+=\prod_j(\sfa'^+_j)^{r_j}=\sfa'\cO_{Y_K^+}$. Let $\bar\fq_l$ be the ideal sheaf of $\bar{g}_l(\bar{Q}_l)$ on $\bar{Y}_l^+$ and $\bar{G}_l\in\cD_{\bar{Y}_l^+}$ the divisor obtained by the blow-up of $\bar{Y}_l^+$ along $\bar\fq_l$. They are base-changed to $\sfq$ on $Y_K^+$ and $G_K\in\cD_{Y_K^+}$.

We see that $\mld_{\eta_{g_K(Q_K)}}(Y_K^+,E_K^+,\sfq^n\sfa'^+)=0$ with $n=\dim F_K-1$ and it is computed by $G_K$. We have $\bar\sfa'(l)_K^+\cO_{Y_K^+}=\prod_j(\sfa'^+_j+\sfI^+_{lj})^{r_j}$ and $\ord_{G_K}\sfa'^+_j=\ord_{S_K}\sfa_j<\ord_{G_K}\sfI_{lj}^+$ from Lemma \ref{lem:weaktf}(\ref{itm:weaktf_K}), so $\mld_{\eta_{g_K(Q_K)}}(Y_K^+,E_K^+,\sfq^n\bar\sfa'(l)_K^+\cO_{Y_K^+})=0$ and it is computed by $G_K$. Then $\mld_{\eta_{\bar{g}_l(\bar{Q}_l)}}(\bar{Y}_l^+,\bar{E}_l^+,\bar\fq_l^n\bar\sfa'(l)^+)=0$ and it is computed by $\bar{G}_l$. We regard $\bar{Y}_l^+$ as a family over $\bar{Q}_l$. There exists a dense open subset $\bar{Q}_l^\circ$ of $\bar{Q}_l$ such that for any closed point $q\in Q_l^\circ=\bar{Q}_l^\circ\times_{\bar{Y}_l}Y_l$ with its image $z\in Z_l$, $\mld_q((Y_l)_z,(E_l)_z,\fm_q^n\bar\sfa'(l)\cO_{(Y_l)_z})=0$, computed by $(G_l)_q$, and $\ord_{(G_l)_q}\bar\sfa'_j(l)\cO_{(Y_l)_z}=\ord_{S_K}\sfa_j$, where $\fm_q$ is the maximal ideal sheaf of $q\in(Y_l)_z$ and $G_l=\bar{G}_l\times_{\bar{Y}_l}Y_l$. The $(G_l)_q$ is obtained by the blow-up of $(Y_l)_z$ at $q$. For $i\in I_l$ with $z=s_l(i)$, $\bar\sfa'(l)\cO_{(Y_l)_z}=\prod_j(\fa'_{ij}+\cI_{lij})^{r_j}$ and $\ord_{(G_l)_q}\bar\sfa'_j(l)\cO_{(Y_l)_z}<\ord_{(G_l)_q}\cI_{lij}$ by Lemma \ref{lem:weaktf}(\ref{itm:weaktf_z}). Applying Theorem \ref{thm:adic0}, we have $\mld_q((Y_l)_z,(E_l)_z,\fm_q^n\fa'_i)=0$, and the log canonicity of $((Y_l)_z,(E_l)_z,\fa'_i)$ about $(Q_l^\circ)_z$ is concluded.
\end{proof}

Theorem \ref{thm:plt} is completed.

\section{The threefold case}\label{sec:threefold}
We shall prove Theorem \ref{thm:acc}. By Remark \ref{rmk:reduction}, the theorem follows from Conjecture \ref{cnj:stability} for $d=3$ with $\mld_{P_K}(X_K,\sfa)>1$. In Remark \ref{rmk:remain}, Conjecture \ref{cnj:stability} is reduced to the case when $(X_K,\sfa)$ is an lc pair which has a minimal lc centre $Z$ of positive dimension. If $d=3$, then by Theorem \ref{thm:centre}, $Z$ is the smallest lc centre and it is normal. If $Z$ is a surface, then one can apply Theorem \ref{thm:plt}. If $Z$ is a curve, then $\mld_{P_K}(X_K,\sfa)\le1$ by Proposition \ref{prp:curve}. Therefore, we obtain Theorem \ref{thm:acc}.

\begin{proposition}\label{prp:curve}
Let $P\in(X,\fa)$ be a germ of an lc pair on a non-singular $R$-variety $X$ of dimension $3$ with $R=K[[x_1,\ldots,x_d]]$ whose smallest lc centre is a curve. Then $\mld_P(X,\fa)\le1$.
\end{proposition}

\begin{proof}
The smallest lc centre $C$ of $(X,\fa)$ is non-singular by Theorem \ref{thm:centre}. Setting $(X_0,\Delta_0,\fa_0):=(X,0,\fa)$ and $C_0:=C$, we build a tower of finitely many blow-ups
\begin{align*}
X_n\to\cdots\to X_i\xrightarrow{f_i}X_{i-1}\to\cdots\to X_0=X
\end{align*}
such that (i) $f_i\colon X_i\to X_{i-1}$ is the blow-up along $C_{i-1}$, (ii) $E_i$ is the exceptional divisor of $f_i$, (iii) $(X_i,\Delta_i,\fa_i)$ is the pull-back of $(X_{i-1},\Delta_{i-1},\fa_{i-1})$, (iv) $C_i$ is a non-singular non-klt centre on $X_i$ of $(X,\fa)$ mapped onto $C_{i-1}$, and (v) $a_{E_i}(X,\fa)>0$ for $i<n$ and $a_{E_n}(X,\fa)=0$. Here one can prove the effectiveness $\Delta_i\ge0$ and the non-singularity of $C_i$ by induction. Indeed, if they hold for $i-1$, then $\ord_{E_i}\Delta_i=\ord_{C_{i-1}}\Delta_{i-1}+\ord_{C_{i-1}}\fa_{i-1}-1>0$ by Lemma \ref{lem:curve}. Unless $a_{E_i}(X,\fa)=0$, an arbitrary lc centre $C_i$ of $(X_i,\Delta_i,\fa_i)$ mapped onto $C_{i-1}$ is a curve and is minimal. The non-singularity of $C_i$ follows from Theorem \ref{thm:centre}.

Let $F$ be the divisor obtained by the blow-up of $X_n$ along a curve in $E_n\cap(f_1\circ\cdots\circ f_n)^{-1}(P)$. Then $a_F(X,\fa)=a_F(X_n,\Delta_n,\fa_n)\le a_F(X_n,E_n)=1$ by $\Delta_n\ge(1-a_{E_n}(X,\fa))E_n=E_n$.
\end{proof}

\begin{lemma}\label{lem:curve}
Let $(X,\fa)$ be a pair on a non-singular $R$-variety $X$ and $Z$ a non-klt centre of $(X,\fa)$. Then $\ord_Z\fa\ge1$. If in addition $\codim_XZ\ge2$, then $\ord_Z\fa>1$.
\end{lemma}

\begin{proof}
The lemma is obvious if $Z$ is a divisor, so we may assume $\codim_XZ\ge2$. Setting $X_0:=X$, $Z_0:=Z$ and $\fa_0:=\fa$, we build a tower of finitely many blow-ups
\begin{align*}
X_n\to\cdots\to X_i\xrightarrow{f_i}X_{i-1}\to\cdots\to X_0=X
\end{align*}
such that (i) $f_i$ is the composition $X_i\xrightarrow{h_i}X'_{i-1}\xrightarrow{g_{i-1}}X_{i-1}$ of the blow-up $h_i\colon X_i\to X'_{i-1}$ along the strict transform $Z'_{i-1}$ on $X'_{i-1}$ of $Z_{i-1}$ and an embedded resolution $g_{i-1}\colon X'_{i-1}\to X_{i-1}$ of singularities of $Z_{i-1}$, in which $g_{i-1}$ is isomorphic outside the singular locus of $Z_{i-1}$, (ii) $E_i$ is the exceptional divisor of $h_i$, (iii) $\fa_i$ is the weak transform on $X_i$ of $\fa_{i-1}$, (iv) $Z_i$ is a non-klt centre on $X_i$ of $(X,\fa)$ mapped onto $Z_{i-1}$, and (v) $a_{E_i}(X,\fa)>0$ for $i<n$ and $a_{E_n}(X,\fa)\le0$.

Supposing $\ord_Z\fa\le1$, we shall derive by induction two inequalities $\ord_{Z_i}\fa_i\le1$ and $a_{E_n}(X_i,\fa_i)\le0$ for any $i$. The claim for $i=0$ is trivial. If they hold for $i-1$, then $\ord_{Z_i}\fa_i\le\ord_{V_i}\fa_i\le\ord_{Z_{i-1}}\fa_{i-1}\le1$ by \cite[Lemmata III.7, III.8]{Hi64} for an irreducible closed subset $V_i$ of $Z_i$ meeting the non-singular locus of $Z_i$ such that $V_i\to Z_{i-1}$ is finite and surjective. Note that the symbol $\nu^{(1)}$ in \cite{Hi64} stands for the order. The triplet $(X_{i-1},0,\fa_{i-1})$ is pulled back to $(X_i,\Delta_i,\fa_i)$ with $\ord_{E_i}\Delta_i=1+\ord_{Z_{i-1}}\fa_{i-1}-\codim_{X_{i-1}}Z_{i-1}\le0$, so $a_{E_n}(X_i,\fa_i)\le a_{E_n}(X_i,\Delta_i,\fa_i)=a_{E_n}(X_{i-1},\fa_{i-1})\le0$.

We obtained $a_{E_n}(X_n,\fa_n)\le0$. However, it contradicts $a_{E_n}(X_n)=1$ and $\ord_{E_n}\fa_n=0$.
\end{proof}

\appendix
\section{Generic limits}\label{sec:appendix}
The generic limit is a limit of ideals. It was constructed first by de Fernex and Musta\c{t}\u{a} \cite{dFM09} using ultraproducts, and then by Koll\'ar \cite{Kl08} using Hilbert schemes. We set $\bar{R}=k[x_1,\ldots,x_d]$ with maximal ideal $\bar\fm$, and $\bA_k^d=\Spec\bar{R}$ with origin $\bar{P}$. We also set $R=k[[x_1,\ldots,x_d]]$ with $\fm=\bar\fm R$, and $X=\Spec R$ with closed point $P$. Mostly we discuss on the spectrum of a noetherian ring, where an ideal in the ring is identified with its coherent ideal sheaf.

We introduce the notion of a family of approximated ideals by which a generic limit is defined.

\begin{definition}\label{def:family}
Let $S=\{(\fa_{i1},\ldots,\fa_{ie})\}_{i\in I}$ be a collection of $e$-tuples of ideals in $R$, indexed by an infinite set $I$. A \textit{family} $\cF$ \textit{of approximations} of $S$ consists of, with $l_0$ fixed, for each $l\ge l_0$,
\begin{enumerate}
\item[(a)]
a variety $Z_l$,
\item[(b)]
an ideal sheaf $\bar\sfa_j(l)$ on $\bA^d_k\times_{\Spec k}Z_l$ containing $\bar\fm^l\otimes_k\cO_{Z_l}$ for $1\le j\le e$,
\item[(c)]
an infinite subset $I_l$ of $I$ and a map $s_l\colon I_l\to Z_l(k)$, where $Z_l(k)$ is the set of $k$-points on $Z_l$, and
\item[(d)]
a dominant morphism $t_{l+1}\colon Z_{l+1}\to Z_l$,
\end{enumerate}
such that
\begin{enumerate}
\item
$\bar\sfa_j(l)$ gives a flat family of closed subschemes of $\bA_k^d$ parametrised by $Z_l$,
\item\label{itm:family_pullback}
the pull-back of $\bar\sfa_j(l)$ by $\id_{\bA_k^d}\times t_{l+1}$ is $\bar\sfa_j(l+1)+\bar\fm^l\otimes_k\cO_{Z_{l+1}}$,
\item\label{itm:family_aij}
$\fa_{ij}+\fm^l=\bar\sfa_j(l)_{s_l(i)} R$ for $i\in I_l$, where $\bar\sfa_j(l)_z$ is the ideal in $\bar{R}$ given by $\bar\sfa_j(l)$ at $z\in Z_l$,
\item
$s_l(I_l)$ is dense in $Z_l$, and
\item
$I_{l+1}\subset I_l$ and $t_{l+1}\circ s_{l+1}=s_l|_{I_{l+1}}$.
\end{enumerate}
\end{definition}

The construction of $\cF$ using Hilbert schemes is exposed in \cite[Section 4]{dFEM10}. In general, there exist essentially different families of approximations.

For a field extension $K$ of $k$, we set $\bar{R}_K=\bar{R}\otimes_kK=K[x_1,\ldots,x_d]$ with $\bar\fm_K=\bar\fm\bar{R}_K$, and $\bA_K^d=\Spec\bar{R}_K$ with origin $\bar{P}_K$. We also set $R_K=\widehat{R\otimes_kK}=K[[x_1,\ldots,x_d]]$ with $\fm_K=\fm R_K$, and $X_K=\Spec R_K$ with closed point $P_K$.

\begin{definition}\label{def:limit}
Suppose that a family $\cF$ of approximations of $S$ is given as in Definition \ref{def:family}. For this $\cF$, take the union $K=\varinjlim_lK(Z_l)$ of the function fields $K(Z_l)$ of $Z_l$ by the inclusions $t_{l+1}^*\colon K(Z_l)\hookrightarrow K(Z_{l+1})$. Then the \textit{generic limit} of $S$ with respect to $\cF$ is the $e$-tuple $(\sfa_1,\ldots,\sfa_e)$ of ideals in $R_K$ such that $\sfa_j+\fm_K^l=\bar\sfa_j(l)_KR_K$ for all $l\ge l_0$, where $\bar\sfa_j(l)_K$ is the ideal in $\bar{R}_K$ given by $\bar\sfa_j(l)$ at the natural $K$-point $\Spec K\to Z_l$.
\end{definition}

\begin{remark}\label{rmk:limit}
We have $\sfa_j=\varprojlim_l\bar\sfa_j(l)_K$, by $\bar\sfa_j(l)_K=\bar\sfa_j(l+1)_K+\bar\fm_K^l$ from (\ref{itm:family_pullback}) in Definition \ref{def:family}.
\end{remark}

\begin{definition}
Let $\cF=(Z_l,(\bar\sfa_j(l))_j,I_l,s_l,t_{l+1})_{l\ge l_0}$ and $\cF'=(Z'_l,(\bar\sfa'_j(l))_j,I'_l,s'_l,
\linebreak
t'_{l+1})_{l\ge l'_0}$ be families of approximations of $S$. A \textit{morphism} $\cF'\to\cF$ consists of dominant morphisms $f_l\colon Z'_l\to Z_l$ for $l\ge l'_0$, with $l'_0\ge l_0$ imposed, such that
\begin{enumerate}
\item
$t_{l+1}\circ f_{l+1}=f_l\circ t'_{l+1}$,
\item
the pull-back of $\bar\sfa_j(l)$ by $\id_{\bA_k^d}\times f_l$ is $\bar\sfa'_j(l)$, and
\item
$I'_l\subset I_l$ and $f_l\circ s'_l=s_l|_{I'_l}$.
\end{enumerate}
An $\cF'$ is called a \textit{subfamily} of $\cF$ if it is equipped with a morphism $\cF'\to\cF$ as above such that all $f_l$ are open immersions.
\end{definition}

We want to compare minimal log discrepancies over $X$ and $X_K$. The comparison of those for approximated ideals is a consequence of the existence of a family of log resolutions on an open subfamily of triplets and Corollary \ref{cor:extension}.

\begin{lemma}[cf.\ {\cite[Proposition 3.2(ii)]{K14}}]\label{lem:approx_mld}
Notation as above. Let $(\sfa_1,\ldots,\sfa_e)$ be the generic limit of $S$ with respect to $\cF$. Then after replacing $\cF$ with a subfamily,
\begin{align*}
\mld_{P_K}(X_K,\prod_j(\sfa_j+\fm_K^l)^{r_j})=\mld_{\bar{P}}(\bA_k^d,\prod_j\bar\sfa_j(l)_z^{r_j})
\end{align*}
for all $r_1,\ldots,r_e>0$ and all $z\in Z_l$.
\end{lemma}

We utilise a projective morphism which is descended to $\bA_K^d$.

\begin{definition}
A projective morphism $f_K\colon Y_K\to X_K$ is said to be \textit{descendible} if there exists a projective morphism $\bar{f}_K\colon\bar{Y}_K\to\bA_K^d$ whose base change to $X_K$ is $f_K$.
\end{definition}

\begin{proposition}\label{prp:descendible}
Let $f_K\colon Y_K\to X_K$ be a projective morphism of $R_K$-varieties which is isomorphic outside $P_K$. Then $f_K$ is descendible.
\end{proposition}

\begin{proof}
Assuming $d\ge1$, $f_K$ is the blow-up along an ideal $\fn_K$ in $R_K$ \cite[Theorem 8.1.24]{Li02}. We may assume $\codim_{X_K}\Cosupp\fn_K\ge2$, then $\Cosupp\fn_K\subset P_K$, that is, $\fn_K$ is an $\fm_K$-primary ideal. Thus, $\fn_K$ is the pull-back of the ideal $\bar\fn_K=\fn_K\cap\bar{R}_K$ in $\bar{R}_K$. Since blowing-up commutes with flat base change \cite[Proposition 8.1.12(c)]{Li02}, the blow-up of $\bA_K^d$ along $\bar\fn_K$ is base-changed to $f_K$.
\end{proof}

Let $f_K\colon Y_K\to X_K$ be a descendible projective morphism, descended to $\bar{f}_K\colon\bar{Y}_K\to\bA_K^d$. This $\bar{f}_K$ is defined over $k(Z_{l'_0})$ for some $l'_0\ge l_0$. For $l\ge l'_0$, one can construct inductively a projective morphism $\bar{f}'_l\colon\bar{Y}'_l\to\bA_k^d\times_{\Spec k}Z'_l$ with a non-singular open subvariety $Z'_l$ of $Z_l$ such that (i) $\bar{Y}'_l$ is flat over $Z'_l$, (ii) $Z'_{l+1}\subset t_{l+1}^{-1}(Z'_l)$, and (iii) $\bar{f}'_{l+1}$ and $\bar{f}_K$ are the base changes of $\bar{f}'_l$, by generic flatness \cite[Corollaire IV.11.1.5]{EGA}. These $Z'_l$ with $I'_l=s_l^{-1}(Z'_l(k))$ form a subfamily $\cF'$ of $\cF$. Replacing $\cF$ with $\cF'$, we obtain a commutative diagram
\begin{align}\label{eqn:cartesian}
\begin{aligned}
\xymatrix@!0@-1.5pt{
Y_K\ar[rrrr]\ar[rrd]\ar[ddd]_{f_K}&&&&Y_l\ar@{-}[d]\ar[rrd]&&\\
&&\bar{Y}_K\ar[rrrr]\ar[ddd]_{\bar{f}_K}&&\ar[dd]_(.25){f_l}&&\bar{Y}_l\ar[ddd]_{\bar{f}_l}\\
&&&&&&\\
X_K\ar@{-}[rr]\ar[rrd]&&\ar[rr]&&X\times_{\Spec k}Z_l\ar[rrd]&&\\
&&\bA_K^d\ar[rrrr]&&&&\bA_k^d\times_{\Spec k}Z_l
}
\end{aligned}
\end{align}
for $l\ge l_0$ (the $l_0$ is replaced) such that (i) $Z_l$ is non-singular, (ii) $\bar{f}_l$ is projective, (iii) $\bar{Y}_l$ is flat over $Z_l$, and (iv) $\bar{f}_{l+1}$, $\bar{f}_K$, $f_l$ and $f_K$ are the base changes of $\bar{f}_l$. In general, $X_K\to X\times_{\Spec k}Z_l$ is not the base change of $\bA_K^d\to\bA_k^d\times_{\Spec k}Z_l$.

Whenever an algebraic object over $X_K$ descendible to $\bA_K^d$ is specified, by taking a subfamily, one can construct (\ref{eqn:cartesian}) so that it comes from a flat family over $Z_l$. For example, suppose that $E_K\in\cD_{X_K}$ with centre $P_K$ is given. It is realised as a divisor on $Y_K$ equipped with a log resolution $f_K\colon Y_K\to X_K$ of $(X_K,\fm_K)$, which is isomorphic outside $P_K$. This $f_K$ is descended to a log resolution $\bar{f}_K$ by Proposition \textup{\ref{prp:descendible}}, and $\bar{f}_K$ is extended to a family $\bar{f}_l$ of log resolutions in (\ref{eqn:cartesian}) by generic smoothness. There exists a prime divisor $\bar{E}_l$ on $\bar{Y}_l$ which is base-changed to $E_K$. By this observation, Lemma \ref{lem:approx_mld} is refined as follows.

\begin{lemma}[cf.\ {\cite[Proposition 3.2(iii)]{K14}}]\label{lem:mld}
Notation as above. Fix $r_1,\ldots,r_e>0$ and $E_K\in\cD_{X_K}$ computing $\mld_{P_K}(X_K,\prod_j\sfa_j^{r_j})$. Then after replacing $\cF$ with a subfamily, there exists a divisor $\bar{E}_l$ over $\bA_k^d\times_{\Spec k}Z_l$ for any $l$, base-changed to $E_K$, such that
\begin{gather*}
\mld_{P_K}(X_K,\prod_j\sfa_j^{r_j})=\mld_{\bar{P}}(\bA_k^d,\prod_j\bar\sfa_j(l)_z^{r_j})=a_{(\bar{E}_l)_z}(\bA_k^d,\prod_j\bar\sfa_j(l)_z^{r_j}),\\
\ord_{E_K}\sfa_j=\ord_{E_K}(\sfa_j+\fm_K^l)=\ord_{(\bar{E}_l)_z}\bar\sfa_j(l)_z<l,
\end{gather*}
for all $z\in Z_l$.
\end{lemma}

We apply the ideal-adic semi-continuity of log canonicity by Koll\'ar, and de Fernex, Ein and Musta\c{t}\u{a}.

\begin{theorem}[{\cite{Kl08}, \cite{dFEM10}, \cite[Proposition 2.20]{dFEM11}}]\label{thm:adic0}
Let $Q\in Y$ be a germ of an lc variety and set $\hat{Y}=\Spec\widehat{\cO_{Y,Q}}$ with closed point $\hat{Q}$. Let $\fa=\prod_j\fa_j^{r_j}$ be an $\bR$-ideal on $\hat{Y}$. Suppose $\mld_{\hat{Q}}(\hat{Y},\fa)=0$ and it is computed by $\hat{E}\in\cD_{\hat{Y}}$. If an $\bR$-ideal $\fb=\prod_j\fb_j^{r_j}$ on $\hat{Y}$ satisfies $\fa_j+\fp_j=\fb_j+\fp_j$ for all $j$, where $\fp_j=\{u\in\cO_{\hat{Y}}\mid\ord_{\hat{E}}u>\ord_{\hat{E}}\fa_j\}$, then $\mld_{\hat{Q}}(\hat{Y},\fb)=0$.
\end{theorem}

\begin{corollary}\label{cor:adic0}
In Lemma \textup{\ref{lem:mld}}, if $\mld_{P_K}(X_K,\prod_j\sfa_j^{r_j})=0$, then $\mld_{P}(X,\prod_j\fa_{ij}^{r_j})=0$ for any $i\in I_l$ on a subfamily. In particular, if $(X_K,\prod_j\sfa_j^{r_j})$ is lc, then so is $(X,\prod_j\fa_{ij}^{r_j})$.
\end{corollary}

{\small
\begin{acknowledgements}
I should like to thank Professors O.~Fujino, Y.~Gongyo, J.~Koll\'ar, M.~Musta\c{t}\u{a} and N.~Nakayama for discussions. The research was partially supported by JSPS Grant-in-Aid for Young Scientists (A) 24684003.
\end{acknowledgements}
}

\end{document}